\documentclass[11pt]{amsart}
\usepackage{amsmath, amsthm, amsfonts, amssymb, color}
\usepackage{mathabx}
\usepackage[pagebackref]{hyperref}
\usepackage{esint}
\usepackage{bbm, dsfont}

\usepackage[left=1.05in, right=1.05in, top=1.05in, bottom=1.05in]{geometry}

\newcommand{\ptwo}{P^2}
\newcommand{\pone}{P^1}
\newcommand{\kone}{k_1}

\newcommand{\tptwo}{\tilde{P}^2}
\newcommand{\lip}{\mathrm{Lip}}
\newcommand{\dtt}{\frac{dt}{t}}

\begin{document}
	
	
\setcounter{secnumdepth}{5}
\numberwithin{equation}{section}
\theoremstyle{plain}
	\newtheorem{theorem}{Theorem}[section]
	\newtheorem{proposition}[theorem]{Proposition}
	\newtheorem{coro}[theorem]{Corollary}
	\newtheorem{lemma}[theorem]{Lemma}
	\newtheorem{definition}[theorem]{Definition}
	
	\newtheorem{assum}{Assumption}[section]
	\newtheorem{example}[theorem]{Example}
	\newtheorem{remark}[theorem]{Remark}
	\renewcommand{\theequation}
	{\thesection.\arabic{equation}}
	
	\def\SL{\sqrt H}
	
	\newcommand{\mar}[1]{{\marginpar{\sffamily{\scriptsize
					#1}}}}
	



	\newcommand\R{\mathbb{R}}
	\newcommand\RR{\mathbb{R}}
	\newcommand\CC{\mathbb{C}}
	\newcommand\NN{\mathbb{N}}
	\newcommand\ZZ{\mathbb{Z}}
	\def\RN{\mathbb{R}^n}
	\renewcommand\Re{\operatorname{Re}}
	\renewcommand\Im{\operatorname{Im}}
	
	\newcommand{\mc}{\mathcal}
	\newcommand\D{\mathcal{D}}
	\def\hs{\hspace{0.33cm}}
	\newcommand{\la}{\lambda}
	\def \l {\lambda}
	\newcommand{\eps}{\varepsilon}
	\newcommand{\pl}{\partial}
	\newcommand{\supp}{{\rm supp}{\hspace{.05cm}}}
	\newcommand{\x}{\times}
	\newcommand{\lag}{\langle}
	\newcommand{\rag}{\rangle}
	
	\newcommand\wrt{\,{\rm d}}
	
	\newcommand {\BMO}{{\mathrm{BMO}}}
	\newcommand {\Rn}{{\mathbb{R}^{n}}}
	\newcommand {\rb}{\rangle}
	\newcommand {\lb}{{\langle}}
	\newcommand {\HT}{\mathcal{H}}
	\newcommand {\Hp}{\mathcal{H}^{p}_{FIO}(\Rn)}
	\newcommand {\ud}{\mathrm{d}}
	\newcommand {\Sp}{S^{*}(\Rn)}
	\newcommand {\Sw}{\mathcal{S}}
	\newcommand {\w}{{\omega}}
	\newcommand {\ph}{{\varphi}}
	\newcommand {\para}{{\mathrm{par}}}
	\newcommand {\N}{{{\mathbb N}}}
	\newcommand {\Z}{{{\mathbb Z}}}
	\newcommand {\F}{{\mathcal{F}}}
	\newcommand {\C}{{\mathbb C}}
	\newcommand {\vanish}[1]{\relax}
	\newcommand {\ind}{{\mathbf{1}}}
	
	\newcommand{\wt}{\widetilde}
	
	\newtheorem{corollary}[theorem]{Corollary}


\newcommand{\czerol}{C_{0, L}}
\newcommand{\czerou}{C_{0, U}}
\newcommand{\conel}{C_{1, L}}
\newcommand{\coneu}{C_{1, U}}
	\def\beq{\begin{equation}}
		\def\endeq{\end{equation}}
	\def\lesim{\lesssim}
	\def\mc{\mathcal}
	\def\pnorm#1{{ \Big(  #1 \Big) }}
	\def\norm#1{{ \Big|  #1 \Big| }}
	\def\Norm#1{{ \Big\|  #1 \Big\| }}
	\def\inn#1#2{\langle#1,#2\rangle}
	
	\def\set#1{{ \left\{ #1 \right\} }}
	\def\vector{\overrightarrow}
	\def\dim{\text{dim}}

	\newcommand{\red}[1]{{\color{red} #1}}

	\newcommand{\blue}[1]{{\color{blue} #1}}

	\newcommand{\sg}[1]{{\bf\color{red}{#1}}}


	\title[Hilbert transform along planar curves]
	{Hilbert transforms along variable planar curves: Lipschitz regularity}

	\author{Naijia Liu \ and \ Haixia Yu}
	
	\address{
		Naijia Liu,
		Department of Mathematics,
		Sun Yat-sen University,
		Guangzhou, 510275,
		P.R.~China}
	\email{liunj@mail2.sysu.edu.cn}
	
	\address{
		Haixia Yu,
		Department of Mathematics,
		Sun Yat-sen University,
		Guangzhou, 510275,
		P.R.~China}
	\email{yuhx26@mail.sysu.edu.cn}

	\subjclass[]{}

	\begin{abstract}
		In this paper, for $1<p<\infty$, we obtain the $L^p$-boundedness of the Hilbert transform $H^{\gamma}$ along a variable plane curve $(t,u(x_1, x_2)\gamma(t))$, where $u$ is a Lipschitz function with small Lipschitz norm, and $\gamma$ is a general curve satisfying some suitable smoothness and curvature conditions.
	\end{abstract}
	
	\maketitle

\section{introduction}
	
Let $u:\ \R^2\to \R$ be a measurable function, $\gamma:\ [-1, 1]\to \R$ with $\gamma(0)=0$ be a smooth function that is either even or odd, and is strictly increasing on $[0, 1]$, we define
\begin{align}
\label{0721-operator-hilbert} H^{\gamma}f(x_1,x_2)={\rm p.\,v.}\int_{-1}^{1}f(x_{1}-t,x_{2}-u(x_1,x_2)\gamma(t))\,\dtt.
\end{align}
The main purpose of this article is to study the $L^{p}$-boundedness of the Hilbert transform $H^{\gamma}$ along the variable plane curve $(t,u(x_1,x_2)\gamma(t))$ defined in \eqref{0721-operator-hilbert}.
	
Our Hilbert transform $H^{\gamma}$ is motivated by Stein's conjecture \cite{MR934224}, which correspond to the limiting case $\gamma(t)=t$ not covered here. The so-called Stein conjecture can be stated as follows: Denote
$$H_{\varepsilon_0}f(x_1,x_2)={\rm p.\,v.}\int_{-\varepsilon_0}^{\varepsilon_0}
f(x_1-t,x_2-u(x_1,x_2)t)\,\frac{dt}{t},
$$
one wishes to know if the operator $H_{\varepsilon_0}$ is bounded on $L^p$ for some $p\in(1,\infty)$, where $u:\ \mathbb{R}^2\rightarrow  \mathbb{R}$ is a Lipschitz function and $\varepsilon_0>0$ small enough depending on $\|u\|_{\textrm{Lip}}$. It should be mentioned that this Lipschitz regularity imposed on $u$ is critical, since a counterexample based on a construction of the Besicovitch-Kakeya set shows that no $L^p$-boundednesss of $H_{\varepsilon_0}$ is possible for $C^\alpha$ regularity with $\alpha<1$, where $1<p<\infty$.

We now state some historical evolutions of this conjecture. For any real analytic function $u$, Stein and Street \cite{MR2880220} obtained the $L^p$-boundedness of $H_{\varepsilon_0}$ for all $1<p<\infty$. For $C^\infty$ vector fields $u$ under certain geometric assumptions, Christ, Nagel, Stein and Wainger \cite{MR1726701} showed that $H_{\varepsilon_0}$ is bounded on $L^{p}$ for all $1<p<\infty$. For $u\in C^{1+\alpha}$ with $\alpha>0$, Lacey and Li \cite{MR2654385} proved the $L^2$-boundedness of $H_{\varepsilon_0}$, conditioning on
the boundedness of the so-called Lipschitz-Kakeya maximal operator. For some other partial results toward this conjecture, we refer to \cite{MR545242,MR1009171,MR2219012}.

In the last decade further model problems, with the additional simplifying condition about $u$ have been considered: Based on the works of Lacey and Li \cite{MR2219012,MR2654385}, Bateman \cite{MR3090145} proved the $L^p$-boundedness of $H_{\varepsilon_0}P_{j}^{2}$, with a bound being independent of $j\in \mathbb{Z}$, where $u(x_1,x_2)=u(x_1,0)$ is only a one-variable measurable function, $P_{j}^{2}$
is the Littlewood-Paley projection applied in the second variable and $1<p<\infty$. Later, under the same condition about $u$, Bateman and Thiele \cite{MR3148061} obtained the $L^p$-boundedness of $H_{\varepsilon_0}$ for all $3/2<p<\infty$. Furthermore, Guo \cite{MR3393679,MR3592519} obtained some similar results under the condition that $u$ is constant along a Lipschitz curve, i.e., a perturbation of Bateman and Thiele's result \cite{MR3090145,MR3148061} in the critical Lipschitz regularity.

The so-called Stein conjecture is very difficult. However, we tend to think that the presence of the curvature in $H_{\varepsilon_0}$ is another direction of considering this outstanding open problem. Therefore, we here extend the variable straight line $(t,u(x_1,x_2)t)$ in $H_{\varepsilon_0}$ to the variable curve $(t,u(x_1,x_2)\gamma(t))$ in $H^{\gamma}$ (see \eqref{0721-operator-hilbert}), and consider the $L^{p}$-boundedness of the operator $H^{\gamma}$.

\bigskip

We now state our main theorem.

\begin{theorem}\label{f19}
			We assume that $\gamma$ satisfies the following conditions:
			\begin{enumerate}\label{curve gamma}
				\item[\rm(i)]
				$\inf_{0<t\le 1}\big|(\gamma'/\gamma'')'(t)\big|>0;$
				\item[\rm(ii)]
				$\min_{j\in \{1, 2\}} \inf_{0<t\le 1}\big|t^j\gamma^{(j)}(t)/\gamma(t)\big|>0;$
				\item[\rm(iii)]
				$\sup_{1\le j\le N} \sup_{0<t\le 1}\big|t^{j}\gamma^{(j)}(t)/\gamma(t)\big|<\infty,$ for a sufficiently large integer $N$.
			\end{enumerate}
			Then for every $u: \R^2\to \R$ with $\|u\|_{\lip}\le 1/(2\gamma(1))$, it holds that
			\begin{equation}
				\|H^{\gamma} f\|_p \lesim_{\gamma, p} \|f\|_p
			\end{equation}
			for every $1<p<\infty$.
\end{theorem}
		
\begin{remark}
In \cite{NjSlHx20}, for $1<p\leq\infty$, Liu, Song and Yu obtained the $L^{p}$-boundedness of the corresponding maximal function $M^{\gamma}$ along the variable curve $(t,u(x_1,x_2)\gamma(t))$, under the same assumptions imposed on $u$ and $\gamma$ as in Theorem \ref{f19}. Therefore, Theorem \ref{f19} should be regarded as a natural continuation of \cite{NjSlHx20}.
\end{remark}

\begin{remark}
The smoothness assumption and the strict monotonicity assumption on $\gamma$ are made to avoid certain technicality. In the assumption $\rm(iii)$ of the above theorem, it is more than enough to take $N=10^3$. Some examples of curves satisfying all of the assumptions in Theorem \ref{f19} can be found in \cite[Example 1.5]{NjSlHx20}. In particular, the non-flat homogeneous curves $[t]^{\alpha}$ satisfy our assumptions in Theorem \ref{f19}, where the notation $[t]^{\alpha}$ stands for either $|t|^{\alpha}$ or \textrm{sgn}$(t)|t|^{\alpha}$, $\alpha\in (0,1)\cup (1,\infty)$. Moreover, one can allow  the function $u$ to have a larger Lipschitz norm if one is willing to replace the truncation $[-1, 1]$ in \eqref{0721-operator-hilbert} by a smaller interval.
\end{remark}

\bigskip

To further collect some historical evolutions of our operator \eqref{0721-operator-hilbert}, let us ignore the truncation $[-1, 1]$ and define the Hilbert transform along the variable curve $(t,u(x_1,x_2)\gamma(t))$ as
$$H^{\gamma}_{\infty}f(x_1,x_2)={\rm p.\,v.}\int^{\infty}_{-\infty}
f(x_1-t,x_2-u(x_1,x_2)\gamma(t))\,\frac{dt}{t}.$$

For the case $u$ is constant, the study of the boundedness properties of $H^{\gamma}_{\infty}$ first appeared in the works of Jones \cite{MR161099} and Fabes and Rivi\`ere \cite{MR209787} for studying the behavior of the constant coefficient parabolic differential operators. Later, the study has been extended to  more general families of curves; see, for example, \cite{MR508453,MR714828,MR813582,MR1046743,MR1296728}. This is a classical area of harmonic analysis, which is directly related to the pointwise convergence of Fourier series and the Calder$\acute{\textrm{o}}$n-Zygmund theory.

For the case $u$ is one-variable function, in \cite{MR1364881}, Carbery, Wainger and Wright first obtained the $L^{p}$-boundedness of $H^{\gamma}_{\infty}$ for all $1<p<\infty$, but with the restriction that $u(x_1,x_2)=x_1$, where $\gamma\in C^{3}(\mathbb{R})$ is either an odd or even convex curve on $(0,\infty)$ satisfying $\gamma(0)=\gamma'(0)=0$ and the quantity $t\gamma''(t)/\gamma'(t)$ is decreasing and bounded below on $(0,\infty)$. Under the same conditions imposed on $\gamma$, Bennett \cite{MR1926840} proved the $L^2$-boundedness of $H^{\gamma}_{\infty}$ under the restriction that $u(x_1,x_2)=P(x_1)$, where $P$ is a polynomial. Some other related results under the same restriction about $u$, we refer to \cite{MR2948242,LiYu21}. In \cite{MR1689214}, Carbery and P\'{e}rez showed the $L^{p}$-boundedness of $H^{\gamma}_{\infty}$ for all $1<p<\infty$ as a generalization of the result in Seeger \cite{MR1258491}, where $u(x_1,x_2)\gamma(t)$ was written as $S(x_1,x_1-t)$, under more restrictive third order assumptions about $S$. A key breakthrough about this one-variable case was made by Guo, Hickman, Lie and Roos in \cite{MR3669936}, where the $L^p$-boundedness of $H^{\gamma}_{\infty}$ was proved for all $1<p<\infty$, and $u(x_1,x_2)=u(x_1,0)$ is only a measurable function and $\gamma(t)=[t]^{\alpha}$ with $\alpha\in (0,1)\cup (1,\infty)$.

For the case $u$ is general two-variable function, Seeger and Wainger \cite{MR2053571} obtained the $L^{p}$-boundedness of $H^{\gamma}_{\infty}$ for all $1<p<\infty$, where $u(x_1,x_2)\gamma(t)$ was written as $\Gamma(x_1,x_2,t)$, under some convexity and doubling hypothesis uniformly in $(x_1,x_2)$. A single annulus $L^p$ estimates for $H^{\gamma}_{\infty}$ for all $2<p<\infty$ were obtained in Guo, Hickman, Lie and Roos \cite{MR3669936}, where $u$ is only a measurable function and $\gamma(t)=[t]^{\alpha}$ with $\alpha\in (0,1)\cup (1,\infty)$. Recently, for the same non-flat homogeneous curve $[t]^{\alpha}$, Di Plinio, Guo, Thiele and Zorin-Kranich \cite{MR3841536} obtained the $L^{p}$-boundedness of the truncated Hilbert transform $H^{\gamma}$ along the variable curve $(t,u(x_1, x_2)[t]^{\alpha})$ if $1<p<\infty$ and $\|u\|_{\textrm{Lip}}$ is small, where the truncation $[-1, 1]$ plays a crucial role.

\bigskip

In the end, let us mention the new ingredients that will be used to prove our Theorem \ref{f19}. It is easy to see that Theorem \ref{f19} is the generalization of Di Plinio, Guo, Thiele and Zorin-Kranich's result \cite{MR3841536} for $H^{\gamma}$ from the special non-flat homogeneous curve $[t]^{\alpha}$ to more general curve $\gamma(t)$. In the non-flat homogeneous curve case $\gamma(t)=[t]^\alpha$, we have the following special property, i.e.,
  \begin{align}\label{special property}
  \gamma(ab)=\gamma(a)\gamma(b), \quad  \textrm{for all}\ \  a, b>0,
 \end{align}
which implies that one can add $v(x,y)(u(x,y)v^{-1}(x,y))^{\alpha/(\alpha-1)}$ in the partition of unity to split these operators considered, where $v(x,y)$ is the largest integer power of $2$ less than $u(x,y)$, see \cite[(2.36)]{MR3841536} for more details. In the general curve case, such a decomposition is not appropriate based on the absence of this property \eqref{special property}. Therefore, we have to split our operators by the classical partition of unity, i.e., $1=\Sigma_{l\in \mathbb{Z}} \phi(2^{l}t)$, see \eqref{201127e6-9}. However, we still will encounter the difficulty of $\gamma(2^{-l}t)\neq \gamma(2^{-l})\gamma(t)$, even though we have used this classical partition of unity. To overcome this difficulty and separate $2^{-l}$ from $\gamma(2^{-l}t)$, we may therefore replace $\gamma(2^{-l}t)$ by $\gamma_l(t)=\gamma(2^{-l}t)/\gamma(2^{-l})$ (see \eqref{c2}) in our proof. In the following Lemma \ref{b44}, we will see that $\gamma_l$ behaves ``uniformly" in the parameter $l$.

On the other hand, in \cite{MR3841536}, the authors need to obtain a local smoothing estimate for
$$\mathcal{A}_{u,t^\alpha} f(x,y)=\int_{-\infty}^{\infty} f(x-ut,y-ut^\alpha)\psi_0(t)\,dt.
$$
However, for the general curve case considered here, since the lack of this property \eqref{special property}, we may not reduce our local smoothing estimate for the corresponding operator $\mathcal{A}_{u,\gamma(t)}$, which is one of the main difficulties in the problem. In this paper, we establish a local smoothing estimate for $A^l_u$ (see \eqref{201210e7.7}) based on Sogge's cinematic curvature condition in \cite{MR4078231}. Compared with $\mathcal{A}_{u,t^\alpha}$, the local smoothing estimate for $A^l_u$ will leads to essential difficulties, since the critical point of the phase function in $\mathcal{A}_{u,t^\alpha}$ is independent of $u$, but the critical point of the phase function in $A^l_u$ will depend on $u$.

Another highlight of this paper is that we generalize the result in Seeger \cite{MR955772} to a general homogeneous type space associated with curve $\gamma$. Because the dilation  $(2^{-l_{0}}\gamma_{l_{0}}(2^{l})^{-1},2^{-k})$ is not homogeneous, we can not use the result in Seeger \cite{MR955772} directly, one novelty is that we discovered a corresponding theory which work in the general homogeneous type space. To prove Lemma \ref{a39}, if one use the Sobolev embedding inequality, it will cause different loss for different values of $k$.  To overcome this difficulty, we thus should make a additional restriction on the doubling constants $C_{0,L}$ and $C_{0,U}$, which will leads to an uniform estimate for different values of $k$. To get a better result, we do not first use the Sobolev embedding inequality and read $\sup_{u\in[1,2]}(\sum_{k\in E_{v}}|\cdot|^{2})^{1/2}$ as a norm in a Banach space and generalize Seeger's result to some functions which take values in a Banach space.

\bigskip

The article is organized as follows. In Section $2$, we obtain several properties about $\gamma$ which we need in the proof of our main theorem. In Section $3$, we study dyadic cubes, dyadic maximal functions and dyadic sharp maximal functions in a general homogeneous type space, which allows us to prove a key Lemma \ref{a13} in Section $4$. In Section $4$, we prove this key Lemma \ref{a13} similar as the one in Seeger \cite{MR955772}, where the underlying space is $\mathbb{R}^{n}$ instead of the general homogeneous type space. This lemma allows us to obtain Lemma \ref{a39}, which further leads to Theorem \ref{f19} if $p>2$. In Section $5$, we prove our main Theorem \ref{f19}. To make the proofs in Section $5$ more clearly, we put the proofs of four lemmas in Section $5$ to Sections $6$, $7$ and $8$.

\bigskip

Throughout this paper, we write $f(s)\lesssim g(s)$ to indicate that $f(s)\leq Cg(s)$ for all $s$, where the constant $C>0$ is independent of $s$ but allowed to depend on $\gamma$, and similarly for $f(s)\gtrsim g(s)$ and $f(s)\simeq g(s)$. $\hat{f}$ denotes the Fourier transform of $f$, and $\check{f}$ is the inverse Fourier transform of $f$. Let $\mathcal{S}$ denote the collection of all Schwartz functions. For any $s\in \R$, we denote by $[s]$ the integer part of $s$. For any set $E$, we use $\chi_E$ to denote the characteristic function of $E$, and $E^{\complement}$ indicates its complementary set.

\section{some properties about $\gamma$}

First of all, without loss of generality, we assume that $\gamma(1)=1$. As a consequence, the function $u$ satisfies $\|u\|_{{\rm Lip}}\le 1/2$.
To simplify our future discussion, we extend the domain of the function $\gamma$ from $[-1, 1]$ to $\R$ by setting
		\begin{equation}\label{201124e2-3}
		\gamma(t)=\gamma(t^{-1})^{-1}
		\end{equation}
for every $|t|\ge 1$. Then we have $\gamma(t)\in C([0,\infty)),$ $\gamma(t)$ is increasing on $[0,\infty)$ and $\gamma([0,\infty))=[0,\infty).$ From the assumption $\rm(ii)$ in Theorem \ref{f19} with $j=2$, we see that $\gamma$ is either convex or concave. These two cases can be handled in exactly the same way, and therefore we without loss of generality assume that $\gamma$ is convex, that is, $\gamma''(t)>0$ for every $0<t\leq1$.
		

	
	\begin{lemma}\label{c1}
		Under the above assumptions on $\gamma$, we have
		\begin{enumerate}\label{curve gamma_z}
			\item[\rm(i)] $\gamma(t)$ satisfies a doubling condition on $\R^{+}$, i.e.,
\begin{equation}
			 \czerol\le \frac{\gamma(2t)}{\gamma(t)}\le \czerou
			\end{equation}
for some $\czerol>1$ and $\czerou>1$ that depend only on $\gamma$;
			\item[\rm(ii)] There exists a positive constant $\conel>1$ such that
			\begin{align*}
				\conel\leq \frac{\gamma'(2t)}{\gamma'(t)}
			\end{align*}
 for all $0< t\leq1/2$.
		\end{enumerate}
		
	\end{lemma}
	\begin{proof}
		
		
		We start with the proof of (i).
		Define $h(t)=\ln \gamma(t)$ for all $t\in(0,1]$, by the conditions (ii) and (iii) of Theorem
		\ref{f19}, we know that there exists $C>1$ such that $th'(t)\in[C^{-1},C]$. Then for all $t\in(0,1/2],$ we have
		\begin{align*}
			\ln \frac{\gamma(2t)}{\gamma(t)}=h(2t)-h(t)=h'(\theta t)t\in [C^{-1}\theta^{-1},C\theta^{-1}]\subset[2^{-1}C^{-1},C]
		\end{align*}
		for some $\theta\in[1,2],$ and therefore
\begin{align}\label{f22}
1<e^{2^{-1}C^{-1}}\le \frac{\gamma(2t)}{\gamma(t)}\le e^{C}
\end{align}
for all $t\in(0,1/2]$. By \eqref{201124e2-3}, we know that \eqref{f22} also holds
for all $t\in[1,\infty)$.
		It remains to consider $t\in[1/2,1]$.
	 By continuity, there exists $t_{0}\in[1/2,1]$ such that $\sup_{t\in[1/2,1]}\gamma(t)\gamma(2^{-1}t^{-1})=\gamma(t_{0})\gamma(2^{-1}t_{0}^{-1})$. Note $t_0$ or $2^{-1}t_0^{-1}$ must less than $3/4$, and $\gamma(1)=1$.
	Then,  $\gamma(t_{0})\gamma(2^{-1}t_{0}^{-1})\leq\gamma(3/4)<1$.
		As a result, for all $t\in[1/2,1]$, we have
		$$\frac{\gamma(2t)}{\gamma(t)}=\frac{1}{\gamma(t)\gamma(2^{-1}t^{-1})}\geq \frac{1}{\gamma(3/4)}>1.$$
Because $\gamma(t)$ is increasing on [0,1], for all $t\in[1/2,1]$, we obtain
		$$\frac{\gamma(2t)}{\gamma(t)}=\frac{1}{\gamma(t)\gamma(2^{-1}t^{-1})}\leq \frac{1}{\gamma(1/2)^{2}}.$$
By setting
\begin{equation}
\czerol=\min\{e^{2^{-1}C^{-1}},\gamma(3/4)^{-1}\}>1 \ \ \textrm{and} \ \ \czerou=\max\{e^{C},\gamma(1/2)^{-2}\}>1,
\end{equation} we can finish the proof of (i). The proof for (ii) is essentially the same as \eqref{f22}, we omit it.
	\end{proof}

Let $l\in \N$, we define
\begin{equation}\label{c2}
    \gamma_{l}(t)=\frac{\gamma(2^{-l}t)}{\gamma(2^{-l})}.
\end{equation}	
In the following Lemma \ref{b44}, we would like to prove some estimates for $\gamma_{l}$ uniformly in $l\in \N$. In the model case $\gamma(t)=t^2$, it always holds that $\gamma_l(t)=t^2$. In the general case, $\gamma$ is not homogeneous anymore, and therefore we need an uniform control over the rescaled versions of $\gamma$.
	
\begin{lemma}\label{b44}
		Let $I=[1/2, 2]$, for $l\in \N$ with $2^l\gg 1$, we then have
		\begin{enumerate}
			\item[\rm(i)] $ |\gamma_{l}^{(k)}(t)|\simeq1$, \ \    for $t\in I$ and $k=0,1,2$;
			
			\item[\rm(ii)] $|\gamma_{l}^{(k)}(t)|\lesssim  1$, \ \ for $t\in I$ and  $3\leq k\leq N$;
			
			\item[\rm(iii)] $|((\gamma_{l}')^{-1})(t)|\simeq 1$, \ \
			for  $t\simeq 1$;
\item[\rm(iv)] $|((\gamma_{l}')^{-1})^{(k)}(t)|\lesssim 1$, \ \
			for  $t\simeq 1$ and $0<k<N$,
		\end{enumerate}
where $(\gamma_{l}')^{-1}$ is the inverse function of $\gamma_{l}'.$ Here the implicit constants depend only on $\gamma$.
	\end{lemma}
	\begin{proof}
We start with the proof of (i).
Because $\gamma_{l}(t)$ is increasing on $[0,1]$, we have
		\begin{align*}
			\gamma_{l}(t)\leq \frac{\gamma(2^{-l+1})}{\gamma(2^{-l})}\leq C_{0,U}\ \ \textrm{and }\ \
			\gamma_{l}(t)\geq \frac{\gamma(2^{-l-1})}{\gamma(2^{-l})}\geq C_{0,U}^{-1}.
		\end{align*}
This finishes the proof of (i) with $k=0.$ Next for $k=1,2,$
		by the conditions (ii) and (iii) of Theorem
		\ref{f19}, we conclude that
		\begin{align}\label{f23}
			|t^{k}\gamma_{l}^{(k)}(t)|=\bigg|\frac{(2^{-l}t)^{k}\gamma^{(k)}(2^{-l}t)}{\gamma(2^{-l})}\bigg|\simeq \bigg|\frac{\gamma(2^{-l}t)}{\gamma(2^{-l})}\bigg|=|\gamma_{l}(t)|.
		\end{align}
		Then we have
$|\gamma_{l}^{(k)}(t)|\simeq |\frac{\gamma_{l}(t)}{t^{k}} |\simeq1.$ This finishes the proof of (i).
		
		Next we prove (ii). By the condition (iii) of Theorem
		\ref{f19}, it follows that
		\begin{align*}
			|t^{k}\gamma_{l}^{(k)}(t)|=\bigg|\frac{(2^{-l}t)^{k}\gamma^{(k)}(2^{-l}t)}{\gamma(2^{-l})}\bigg|\lesssim \bigg|\frac{\gamma(2^{-l}t)}{\gamma(2^{-l})}\bigg|=|\gamma_{l}(t)|,
		\end{align*}
		which further leads to
$|\gamma_{l}^{(k)}(t)|\lesssim|\gamma_{l}(t)|t^{-k}\simeq1.$
		
		Finally, we prove (iii) and (iv). The case $k=0$ follows from the doubling property of $\gamma'$ in Lemma \ref{c1}.  The case $k>0$ follows from (i), (ii) and (iii) via elementary computations involving inverse functions, and we leave out the details.
	\end{proof}

	\begin{lemma}\label{a21}
	 Under the above assumptions on $\gamma$, we obtain
		\begin{enumerate}
			\item[\rm(i)] For each $C>0,$ we have
			\begin{align}\label{c40}
				\gamma^{-1}(C t)\lesssim_C \gamma^{-1}(t)
			\end{align}
  for all $t\geq0$;
		
			\item[\rm(ii)] There exist constants $c_{1},c_{2}>0$
			so that
			\begin{align}\label{21-312-210}
				\gamma^{-1}(2^{-k})\gamma^{-1}(2^{-j})\lesssim 2^{-c_{1}(k+j)}
			\end{align}
 for all $j,k\in \mathbb{Z}$ with $j+k\geq0,$
and
			\begin{align}\label{21-312-211}
	2^{-c_{2}(k+j)}\lesssim\gamma^{-1}(2^{-k})\gamma^{-1}(2^{-j})
			\end{align}
for all $j,k\in \mathbb{Z}$ with $j+k\leq0,$
			where $c_{1},c_{2}$ depend only on $C_{0,U}$.
		\end{enumerate}
		
	\end{lemma}
	\begin{proof}
	The first item (i) follows immediately from the doubling property of $\gamma$. We will therefore only present the proof of (ii). Firstly, we will prove $\gamma(2^{j})\gamma(2^{k})\geq C_{0,U}^{j+k}$ for all $j,k\in \mathbb{Z}$ satisfying $\gamma(2^{j})\gamma(2^{k})\leq1$.  It suffices to show $\gamma(2^{j})\gamma(2^{k})\geq \min\{C_{0,L}^{j+k},C_{0,U}^{j+k}\}$ for all $j,k\in \mathbb{Z}$. Furthermore, by symmetry, we should only to show the above inequality for $j\geq k$. There are the following four cases:

If $j,k\geq 0$, we have $\gamma(2^{j})\gamma(2^{k})\geq C_{0,L}^{j}C_{0,L}^{k}\gamma(1)^{2}=C_{0,L}^{j+k}$;

		If $j,k\leq 0$, we have $\gamma(2^{j})\gamma(2^{k})\geq C_{0,U}^{j}C_{0,U}^{k}\gamma(1)^{2}=C_{0,U}^{j+k}$;

If $j\geq 0$, $k\leq 0$ and $j+k\geq0$, we have
		$\gamma(2^{j})\gamma(2^{k})=\frac{\gamma(2^{k+j}2^{-j})}{\gamma(2^{-j})}
		\geq C_{0,L}^{k+j};$
		
If $j\geq 0$, $k\leq 0$ and $j+k\leq0$, we have
		$\gamma(2^{j})\gamma(2^{k})=\frac{\gamma(2^{k+j}2^{-j})}{\gamma(2^{-j})}
		\geq C_{0,U}^{k+j}.$
		
Altogether, we have shown $\gamma(2^{j})\gamma(2^{k})\geq C_{0,U}^{j+k}$ holds for all $j,k\in \mathbb{Z}$ satisfying $\gamma(2^{j})\gamma(2^{k})\leq1$.
		Then, for all $a,b>0$ and $\gamma(a)\gamma(b)\leq1$, we can choose $j,k\in \mathbb{Z}$ such that $a\in[2^{j},2^{j+1})$ and $b\in[2^{k},2^{k+1})$. Furthermore, from the fact that $\gamma(t)$ is increasing on $[0,\infty),$ we conclude that
		\begin{align*}
			\gamma(a)\gamma(b)&\geq\gamma(2^{j})\gamma(2^{k})\geq C_{0,U}^{j+k}
			\simeq C_{0,U}^{j+k+2}
			=(2^{j+1}2^{k+1})^{\log_{2}C_{0,U}}
			\geq (ab)^{\log_{2}C_{0,U}}.
		\end{align*}
		Therefore, for all $a,b>0$ and $\gamma(a)\gamma(b)\leq1,$ we obtain
		$ab\lesssim (\gamma(a)\gamma(b))^{c_{1}}$
		with $c_{1}=\log_{C_{0,U}}2$.
		Let $a=\gamma^{-1}(2^{-k})$ and $b=\gamma^{-1}(2^{-j})$, for all $j,k\in \mathbb{Z}$ and $j+k\geq0,$ we then have
		$\gamma^{-1}(2^{-k})\gamma^{-1}(2^{-j})\lesssim 2^{-c_{1}(k+j)}.$ The proof of \eqref{21-312-211} is similar as \eqref{21-312-210} and we omit it. This finishes the proof of (ii).
	\end{proof}
		
	\begin{definition}\label{a38}
		Let $x=(x_1,x_2)$, we define a function  $\rho(x)$ on $x\in \mathbb{R}^{2}$ in the following way: For all $x\in \mathbb{R}^{2}$,
		\begin{equation*}
			\rho(x)=\begin{cases}\gamma(|x_{1}|^{-1})^{-1}+|x_{2}|, &x_{1}\neq0; \\ |x_{2}|, &x_{1}=0. \end{cases}
		\end{equation*}
		For all $r>0$ and $x_{0}\in \mathbb{R}^{2}$, we also define $B(x_{0},r)=\{x\in \mathbb{R}^{2}:\ \rho(x-x_{0})< r\}$ as a ball centered at $x_{0}$ with radius $r>0$ associated with $\rho$.
	\end{definition}
	We have the following lemma, which says that the function $\rho$ gives rise to a quasi-distance function on $\R^2$.
	\begin{lemma}\label{a12}
One has that the following quasi-triangle inequality
		\begin{align}\label{a11}
			\rho(x+y)\leq C_{0,U}(\rho(x)+\rho(y))
		\end{align}
	holds for all $x,y\in \mathbb{R}^{2}$. Moreover, we have the following doubling condition
		\begin{align}\label{c7}
			|B(x,2r)|\lesssim |B(y,r)|
		\end{align}
holds for all $r>0$ and $x,y\in \mathbb{R}^{2}$. As a result, $(\mathbb{R}^{2},\rho,dx)$ is a homogeneous type space.
		
	\end{lemma}
\begin{proof}
We first prove \eqref{a11}. Let $x=(x_1, x_2)$ and $y=(y_1, y_2)$, by Definition \ref{a38}, it suffices to show
$$\gamma(|x_{1}+y_{1}|^{-1})^{-1}\leq C_{0,U}\big( \gamma(|x_{1}|^{-1})^{-1}+\gamma(|y_{1}|^{-1})^{-1}\big).$$
By monotonicity of $\gamma$, it is enough to prove
$$\gamma\big((|x_{1}|+|y_{1}|)^{-1}\big)^{-1}\leq C_{0,U}\big( \gamma(|x_{1}|^{-1})^{-1}+\gamma(|y_{1}|^{-1})^{-1}\big).$$
Without loss of generality, we assume $|x_{1}|\geq|y_{1}|$. Then
\begin{align*}
\gamma\big((|x_{1}|+|y_{1}|)^{-1}\big)^{-1}&\leq \gamma\big((2|x_{1}|)^{-1}\big)^{-1}\leq C_{0,U} \gamma\big(|x_{1}|^{-1}\big)^{-1}
\leq C_{0,U}\big(\gamma(|x_{1}|^{-1})^{-1}+\gamma(|y_{1}|^{-1})^{-1}\big),
\end{align*}
where we used the doubling property of $\gamma$ in Lemma \ref{curve gamma_z}. This finishes the proof of \eqref{a11}.

Next we show \eqref{c7}.
		Without loss of generality, we assume $x=y=0$.
		Fix $z\in \mathbb{R}^{2}$ satisfying $\rho(z)<2r$, it suffices to show there exists a constant $C$ such that
		$\rho(z/C)< r.$
		Note that
		$\rho(z/2^{j})=\gamma(2^{j}|z_{1}|^{-1})^{-1}+2^{-j}|z_{2}|,$ by the doubling property of $\gamma$ in Lemma \ref{curve gamma_z}, we have $\gamma(2^{j}|z_{1}|^{-1})\geq C_{0,L}^{j}\gamma(|z_{1}|^{-1}).$
		Then
		$$\rho(z/2^{j})\leq C_{0,L}^{-j}\gamma(|z_{1}|^{-1})^{-1}+2^{-j}|z_{2}|\leq \max\{C_{0,L}^{-j},2^{-j}\}\rho(z).$$
		Choose $j$ satisfying $\max\{C_{0,L}^{-j},2^{-j}\}\leq 1/2$, which leads to $\rho(z/2^{j})\leq \frac{1}{2}\rho(z)< r.$
	\end{proof}

	\section{dyadic cubes}
	To prove Lemma \ref{a13} in Section 4, we need to obtain Lemma \ref{a14}. To prove Lemma \ref{a14}, we need to construct dyadic cubes on the homogeneous type space $(\mathbb{R}^{2},\rho,dx)$. So we need first the following lemma:
	\begin{lemma}[\cite{MR1096400}]\label{a27}
		Let $(\mathbb{R}^{2},\rho,dx)$ be a homogeneous type space as in Lemma \ref{a12}. Then
		there exist an index set $I_k$ with $k\in \Z$, a collection of open subsets $\{Q^{k}(\alpha)\subset \mathbb{R}^{2}:\ k\in \mathbb{Z}, \alpha\in I_{k}\}$, and constants $\delta\in(0,1)$ and $ C_{\rho}, c_{\rho}>0$, such that
		\begin{enumerate}
			\item[\rm(i)]  For all $k\in \mathbb{Z},$
			$$\bigg|\bigg(\bigcup_{\alpha}Q^{k}(\alpha)\bigg)^{\complement}\bigg|=0; $$
			\item[\rm(ii)] If $l\geq k,$ then for all $\alpha\in I_{k},\beta\in I_{l}$, we have
			$Q^{l}(\beta)\subset Q^{k}(\alpha)$ or $Q^{l}(\beta)\bigcap Q^{k}(\alpha)=\varnothing;$
			\item[\rm(iii)] For all $(k,\alpha)$ and $l<k$, there exists a unique $\beta$ such that $Q^{k}(\alpha)\subset Q^{l}(\beta);$
			\item[\rm(iv)] $\textrm{Diameter}~(Q^{k}(\alpha))\leq C_{\rho}\delta^{k}$ for every $\alpha\in I_k$ and every $k\in \Z$;
			\item[\rm(v)] Each $Q^{k}(\alpha)$ contains a ball of radius $c_{\rho}\delta^{k}$.
		\end{enumerate}
	\end{lemma}
%

		Let $Q=Q^k(\alpha)$ be a dyadic cube constructed in Lemma \ref{a27}. We will use $r=r(Q)$ to denote $k \log_2 \delta$. Moreover, denote $Q^{*}=B(z_{\alpha}^{k},2C_{0,U}C_{\rho}\delta^{k})$, where $z^k_{\alpha}$ is an arbitrary point in $Q$. The choice of the parameters in $Q^*$ guarantees that $Q\subset Q^*$. For later use, we collect a few properties about the dyadic cubes constructed above.
%
%

	\begin{lemma}\label{d5}
		For every dyadic cube $Q$, the following hold:
		\begin{enumerate}
			\item[\rm(i)]  If $x\in Q$ and $y\in (Q^{*})^{\complement},$ then $\rho(y-x)\gtrsim 2^{r(Q)}$;
			\item[\rm(ii)] $|Q^{*}|\lesssim |Q|$;
			\item[\rm(iii)] $|Q^{k-1}(\alpha)|\lesssim |Q^{k}(\beta)|\ $ for all $k\in \Z,$ $\alpha\in I_{k-1}$ and $\beta\in I_{k}$.
		\end{enumerate}
	\end{lemma}
	\begin{proof}
Recall that the definition of $Q^*$ and the notation $z^{k}_{\alpha}$, by Lemma \ref{a12}, we have
		\begin{align*}
			\rho(y-x)\geq \frac{1}{C_{0,U}}\rho(y-z^k_{\alpha})-\rho(x-z^k_{\alpha})\geq\frac{1}{C_{0,U}}2C_{0,U}C_{\rho}\delta^{k}-C_{\rho}\delta^{k}\geq C_{\rho}\delta^{k}=C_{\rho}2^{r(Q)}.
		\end{align*}
		This finishes the proof of (i). The other two statements follow immediately from \eqref{c7} and the properties (iv) and (v) in Lemma \ref{a27}.
%
	\end{proof}

		Let $(\mathbb{R}^{2},\rho,dx)$ be a homogeneous type space as in Lemma \ref{a12} and $B$ be a Banach space with norm $\|\cdot \|$. Suppose $(X,d\mu)$ is a measure space and let $1\leq p<\infty$, we define $L^{p}(X,d\mu,B)$ as the collection of all $B$-valued measurable function $f$ satisfying $\int_X \|f(x)\|^{p}\,d\mu(x)<\infty$.
	Then we define
	\begin{align*}
		\|f\|_{L^{p}(X,d\mu,B)}=\bigg(\int_X \|f(x)\|^{p}\,d\mu(x)\bigg)^{1/p}
	\end{align*}
	and define $L^{\infty}(X,d\mu,B)$ by the usual sense. We define $L_{\textrm{loc}}^{1}(\mathbb{R}^{2},B)$ as the space of functions that lie in $L^{1}(K,B)$, where $K$ takes over all compact sets in $\mathbb{R}^{2}$. We ignore the measure $d\mu$ so that we always use notation $L^{p}(X,B)$.
	We define $C(X,B)$ as the collection of all $B$-valued measurable continuous function $f$ on $X$. For
 $f\in L_{\textrm{loc}}^{1}(\mathbb{R}^{2},B)$, we define
		\begin{align}\label{c8}
			M_{\textrm{dyad}}f(x)=\sup_{x\in Q}\fint_{Q}\|f(y)\| \,dy,
		\end{align}
		where the supremum takes over all dyadic cubes in Lemma \ref{a27}. In the rest of this section, we state a few useful lemmas; their proofs are standard and are therefore left out. One can read \cite{MR1232192} for details.

	\begin{lemma}\label{c25}
		Let $f\in C(\mathbb{R}^{2},B)$, then we have
		\begin{align*}
			\textrm{$\|f(x)\|\leq M_{\textrm{dyad}}f(x), \ $ a.e. \ $x\in \mathbb{R}^{2}$.}
		\end{align*}
	\end{lemma}
	To prove the following Lemma \ref{a14}, we also need the following Calder$\acute{\textrm{o}}$n-Zygmund decomposition.
	\begin{lemma}\label{a28}
Let $f\in L^{p}(\mathbb{R}^{2},B) $ for some $1\leq p<\infty$. Suppose $\alpha>0$ and define
		\begin{align*}
			\Omega_{\alpha}=\{x\in \mathbb{R}^{2} :\ M_{\textrm{dyad}}f(x)>\alpha\}.
		\end{align*}
Then there exist a sequence of disjoint dyadic cubes $\{Q_{j}\}$ as in Lemma \ref{a27} such that, for all $j$,
		\begin{enumerate}
			\item[\rm(i)]  $Q_{j}\subset \Omega_{\alpha}$ and
			\begin{align}\label{c12}
				\Omega_{\alpha}=\bigcup_{j}Q_{j};
			\end{align}
			\item[\rm(ii)]
			\begin{align}
				\alpha<\fint_{Q_{j}} \|f\|\,dx\lesssim\alpha.
			\end{align}
		\end{enumerate}
	\end{lemma}
	\begin{lemma}\label{d1}
Let $\vec{T}$ be a sublinear operator mapping the space $L^{p_{0}}(\mathbb{R}^{2},B)+L^{\infty}(\mathbb{R}^{2},B)$ to the space of all measurable functions on $\mathbb{R}^{2}$, where $1\leq p_{0}<\infty$. Suppose that
		\begin{align}\label{f27}
			\|\vec{T}f\|_{\infty}\lesssim\|f\|_{\infty} \text{\ \ for all\ } f\in L^{\infty}(\mathbb{R}^{2},B),
		\end{align}
and
		\begin{align}\label{f28}
			\|\vec{T}f\|_{p_{0},\infty}\lesssim \|f\|_{p_{0}} \text{\ \ for all\ } f\in L^{p_{0}}(\mathbb{R}^{2},B).
		\end{align}
		Then, for all $p_{0}<p<\infty$, we have
		\begin{align*}
			\|\vec{T}f\|_{p}\lesssim \|f\|_{p}  \text{\ \ for all\ } f\in L^{p}(\mathbb{R}^{2},B).
		\end{align*}
	\end{lemma}

By Lemmas \ref{a28} and \ref{d1} and a trivial $L^{\infty}$ estimate, we have
	\begin{lemma}\label{c32}
		Let $B$ be a Banach space and $M_{\textrm{dyad}}f$ be defined as in \eqref{c8}, we then have
		\begin{align}\label{c19}
			\big\|M_{\textrm{dyad}}f\big\|_{p}\lesssim \|f\|_{p}
		\end{align}
 for all $1<p<\infty$ and $f\in L^{p}(\mathbb{R}^{2},B)$.
	\end{lemma}
%

	   Let $\rho$ be as in Definition \ref{a38} and $B$ be a Banach space, for
 $f\in L_{\textrm{loc}}^{1}(\mathbb{R}^{2},B)$, we define $M^{\sharp}_{\textrm{dyad}}f$ by
		$$M^{\sharp}_{\textrm{dyad}}f(x)=\sup_{x\in Q}\fint_{Q} \big\|f(y)-f\big|_{Q}\big\|\,dy,$$
		where $Q$ takes over all dyadic cubes as in Lemma \ref{a27} and $f\big|_{Q}=\fint_{Q} f(y)\,dy.$

	\begin{lemma}\label{a14}
		Let $1<p<\infty$ and B be a Banach space. Then one has
		\begin{align*}
			\|f\|_{p}\lesssim \big\|M^{\sharp}_{\textrm{dyad}}f\big\|_{p}
		\end{align*}
		for all $f\in L^{p}(\mathbb{R}^{2},B)\bigcap C(\mathbb{R}^{2},B)$.
	\end{lemma}

Based on these lemmas have been showed in this section, the proof of Lemma \ref{a14} can be found in \cite{MR447953}.

	\section{a key vector valued estimate}
	This section is devoted to proving Lemma \ref{a13}. We follow the spirit of the proof in \cite{MR955772}, where the underlying space is $\mathbb{R}^{n}$ instead of the general homogeneous type space. Working with this general space will brings some extra technical difficulties. We first state a few useful lemmas before proving the key Lemma \ref{a13}.

Throughout this section, we let $2<p<\infty$, $k,v\in \mathbb{Z}$ and $l\in \mathbb{N}$. We also let $a_{k}(u,\xi)\in C^{\infty}(\mathbb{R}^{3})$ and $\supp a_{k}\subset I\times K_{1}\times K_{2}$, where $I$ is a compact subset of $\mathbb{R}$ and $K_{1},K_{2}$ are compact subsets of $\mathbb{R}\setminus0$. 		Define $m_{k}(u,\xi)=a_{k}(u,2^{-l_{0}}\gamma_{l_{0}}(2^{l})^{-1}\xi_{1},2^{-k}\xi_{2})$,
		where $l_{0}\in \mathbb{Z}$ satisfies $2^{k+v}\gamma(2^{-l_{0}})\simeq1$. There maybe multiple choices of $l_{0}$, we here pick an arbitrary one. We let $Q_{x}$ be a dyadic cube in Lemma \ref{a27} satisfying $x\in Q_{x}$, the definitions of $r_{x}=r(Q_{x})$ and $Q_{x}^{*}$ can be found in the previous section.

	\begin{lemma}\label{a26}
	Let us set $v=0$, suppose that there exists a constant $F$ such that $a_{k}$ satisfies
\begin{align}\label{21-314-41}
\bigg\|\sup_{u\in I}|a_{k}(u,D)f|\bigg\|_{q}\leq F\|f\|_{q}
\end{align}
for $q\in\{2,\infty\}$ and $k\in \mathbb{Z}$. Define
		$$S_{1}f(x)=\fint_{Q_{x}} \sup_{u\in I}\bigg(\sum_{|k+r_{x}|\leq N}\bigg|m_{k}(u,D)f(y)-\fint_{Q_{x}}m_{k}(u,D)f(z)\,dz\bigg|^{2}\bigg)^{1/2}\,dy.$$
		Then for all $N\geq1$, we have
\begin{align}\label{21-314-42}
\|S_{1}f\|_{p}\lesssim N^{1/2-1/p}F\|f\|_{p}.
\end{align}
	\end{lemma}
	\begin{proof}
		By \eqref{21-314-41} and scaling, for $q\in\{2,\infty\}$, we conclude that
		$$\bigg\|\sup_{u\in I}\big|m_{k}(u,D)f\big|\bigg\|_{q}\leq F\|\tilde{P}_{k}^{2}f\|_{q},$$
		where $\tilde{P}_{k}^{2}f=(\tilde{\phi}(2^{-k}\xi_{2})\widehat{f})^\vee$ for some $\tilde{\phi}\in C_{c}^{\infty}(\mathbb{R}\setminus0)$ and $\tilde{\phi}=1$ on $K_{2}$.
		
		Define $\vec{T}f=\{m_{k}(u,D)f\}_{k}$, then by Lemma \ref{c32} we have
		\begin{align}\label{b10}
			\|S_{1}f\|_{2}&\lesssim \big\|M^{\sharp}_{\textrm{dyad}}\vec{T}f\big\|_{2}\lesssim \big\|M_{\textrm{dyad}}\vec{T}f\big\|_{2}\lesssim \big\|\vec{T}f\big\|_{L^{2}(L^{\infty}l^{2})}\nonumber\\
			&\leq \bigg(\sum_{k}\bigg\|\sup_{u\in I}\big|m_{k}(u,D)f\big|\bigg\|_{2}^{2}\bigg)^{1/2}\leq
			F\bigg(\sum_{k}\big\|\tilde{P}_{k}^{2}f\big\|_{2}^{2}\bigg)^{1/2}\lesssim F\|f\|_{2},
		\end{align}
		and
		$$\sup_{x}|S_{1}f(x)|\lesssim N^{1/2}\sup_{k}\sup_{u\in I}\sup_{y}\big|m_{k}(u,D)f(y)\big|\leq N^{1/2}F\|f\|_{\infty}.$$
		By interpolation, \eqref{21-314-42} follows.
	\end{proof}
	To prove Lemma \ref{a19}, we need the following estimate:
	\begin{lemma}\label{a20}
Let us set $v=0$, suppose that there exists a constant $G$ such that $a_{k}$ satisfies
\begin{align}\label{f24}
		\big|\partial_{\xi}^{\alpha}a_{k}(u,\xi)\big|\leq G
\end{align}
 for all $u\in I,$ $\xi\in K_{1}\times K_{2}$, $|\alpha|\leq 6$ and $k\in \mathbb{Z}$.	Define
		$$E_{k}^{u}(x,y,z)=\int_{(Q_{x}^{*})^{\complement}} \big|(m _{k}(u,\cdot))^\vee(y-w)-(m _{k}(u,\cdot))^\vee(z-w)\big|\,dw.$$
		Then there exists a constant $\delta>0$ such that
		\begin{align}\label{f2}
			E_{k}^{u}(x,y,z)\lesssim GC_{0,U}^{l}2^{-\delta|k+r_{x}|}
		\end{align}
 for all $y,z\in Q_{x}$ and $u\in I.$
	\end{lemma}
	\begin{proof}
		Define $K_{k}(x)=a _{k}(u,\cdot)^\vee(x)$, then
		$$m _{k}(u,\cdot)^\vee(x)=\gamma^{-1}(2^{-k})^{-1}\gamma_{l_{0}}(2^{l})2^{k}K_{k}\big(\gamma^{-1}(2^{-k})^{-1}\gamma_{l_{0}}(2^{l})x_{1},2^{k}x_{2}\big).$$
	By \eqref{f24} and integration by parts, we obtain a kernel estimate for $K_{k}$ and then we have
		\begin{align}\label{a24}
			\big|m _{k}(u,\cdot)^\vee(x)\big|\lesssim G\gamma^{-1}(2^{-k})^{-1}\gamma_{l_{0}}(2^{l})2^{k}\big(1+|\gamma^{-1}(2^{-k})^{-1}\gamma_{l_{0}}(2^{l})x_{1}|+|2^{k}x_{2}|\big)^{-4}.
		\end{align}

		We prove \eqref{f2} by considering two cases $k+r_{x}\geq0$ and $k+r_{x}\leq0$. For the first case $k+r_{x}\geq0$,
		by Lemmas \ref{d5} and \eqref{a24}, we bound $E_{k}^{u}(x,y,z)$ by $\int_{\rho(w)\geq C2^{r_{x}}}|m _{k}(u,\cdot)^\vee(w)|\,dw$, which can be further bounded by
		\begin{align}\label{f26}
			G\gamma^{-1}(2^{-k})^{-1}\gamma_{l_{0}}(2^{l})2^{k}
\int_{\gamma(|w_{1}|^{-1})^{-1}+|w_{2}|\geq C2^{r_{x}}}\big(1+|\gamma^{-1}(2^{-k})^{-1}\gamma_{l_{0}}(2^{l})w_{1}|+|2^{k}w_{2}|\big)^{-4}\,dw.
		\end{align}
Next, for $\gamma(|w_{1}|^{-1})^{-1}+|w_{2}|\geq C2^{r_{x}}$, we will show
\begin{align}\label{f25}
\big(1+|\gamma^{-1}(2^{-k})^{-1}w_{1}|+|2^{k}w_{2}|\big)^{-1}\lesssim 2^{-\min\{c_{1},1\}(k+r_{x})}.
\end{align}
Here $c_{1}$ is the same as that in Lemma \ref{a21}. Then by \eqref{f26}, \eqref{f25} and an $L^{1}$ estimate, we have verified \eqref{f2} if $k+r_{x}\geq0.$
		To show \eqref{f25}, we consider $\gamma(|w_{1}|^{-1})^{-1}\geq \frac{C}{2}2^{r_{x}}$ or $|w_{2}|\geq \frac{C}{2}2^{r_{x}}$, respectively.
If $\gamma(|w_{1}|^{-1})^{-1}\geq \frac{C}{2}2^{r_{x}}$, we have $|w_{1}|^{-1}\leq \gamma^{-1}(\frac{2}{C}2^{-r_{x}})$ and
\begin{align*}	\big(1+|\gamma^{-1}(2^{-k})^{-1}w_{1}|+|2^{k}w_{2}|\big)^{-1}&\lesssim \gamma^{-1}(2^{-k})|w_{1}|^{-1}\lesssim \gamma^{-1}(2^{-k})\gamma^{-1}\big(\frac{2}{C}2^{-r_{x}}\big)\\
&\lesssim\gamma^{-1}(2^{-k})\gamma^{-1}(2^{-[r_{x}]-1})\lesssim 2^{-c_{1}(k+r_{x})},
\end{align*}
		where in the last inequality we used Lemma \ref{a21}.
If $|w_{2}|\geq \frac{C}{2}2^{r_{x}}$, similarly, we have
$$\big(1+|\gamma^{-1}(2^{-k})^{-1}w_{1}|+|2^{k}w_{2}|\big)^{-1}\lesssim 2^{-(k+r_{x})}.$$
Then \eqref{f25} follows and thus we finish the proof of \eqref{f2} if $k+r_{x}\geq0$.
		
		Next we consider the case $k+r_{x}\leq0.$ By the mean value theorem, we can write
		$$\big|m_{k}(u,\cdot)^\vee(y-w)-m_{k}(u,\cdot)^\vee(z-w)\big|$$
		as the sum of the following two terms:	
		$$\gamma^{-1}(2^{-k})^{-1}\gamma_{l_{0}}(2^{l})(y_{1}-z_{1})\int_{0}^{1}m' _{k}(u,\cdot)^\vee(\theta(y-z)+z-w)\,d\theta,$$
and
$$2^{k}(y_{2}-z_{2})\int_{0}^{1}m'' _{k}(u,\cdot)^\vee(\theta(y-z)+z-w)\,d\theta,$$
where $$\partial_{x_{1}}m _{k}(u,\cdot)^\vee(x)=\gamma^{-1}(2^{-k})^{-1}\gamma_{l_{0}}(2^{l})(m _{k}'(u,\cdot))^\vee(x)\ \ \textrm{and} \ \ \partial_{x_{2}}m _{k}(u,\cdot)^\vee(x)=2^{k}(m _{k}''(u,\cdot))^\vee(x)$$ for some $m_{k}',m_{k}''$ that share the same properties as $m_{k}$.
		After integrating in $w$, for all $k+r_{x}\leq0$, we find that
		\begin{align*}
			E_{k}^{u}(x,y,z)&\lesssim G\big(\gamma^{-1}(2^{-k})^{-1}\gamma_{l_{0}}(2^{l})|y_{1}-z_{1}|+2^{k}|y_{2}-z_{2}|\big)\\
			&\lesssim \gamma_{l_{0}}(2^{l})G\big(2^{c_{2}(k+r_{x})}+2^{k+r_{x}}\big)\lesssim C_{0,U}^{l}G2^{\min\{c_{2},1\}(k+r_{x})},
		\end{align*}
	where in the first inequality we used
 $$\int_{\mathbb{R}^{2}}\big|m' _{k}(u,\cdot)^\vee(w)\big|\,dw\lesssim G \ \ \textrm{ and} \ \ \int_{\mathbb{R}^{2}}\big|m'' _{k}(u,\cdot)^\vee(w)\big|\,dw\lesssim G,$$ which can be proved by a kernel estimate like \eqref{a24}. In the second inequality we used Lemma \ref{a21} and the doubling property of $\gamma$ in Lemma \ref{c1}.
		Then we have verified \eqref{f2} if $k+r_{x}\leq0.$
	\end{proof}
	\begin{lemma}\label{a19}
Let us set $v=0$, suppose that there exist constants $F$ and $G$ such that $a_{k}$ satisfies
		$$\bigg\|\sup_{u\in I}\big|a_{k}(u,D)f\big|\bigg\|_{2}\leq F\|f\|_{2} \ \ \textrm{and}\ \ |\partial_{\xi}^{\alpha}a_{k}(u,\xi)|\leq G$$
 for all $u\in I,$ $\xi\in K_{1}\times K_{2}$, $|\alpha|\leq 6$ and $k\in \mathbb{Z}$.
		Furthermore, let us set $N=\frac{1}{\delta}\log_{2}(C_{0,U}^{l}G/F)$, and define
		$$S_{2}f(x)=\fint_{Q_{x}} \sup_{u\in I}\bigg(\sum_{|k+r_{x}|> N}\bigg|m_{k}(u,D)f(y)-\fint_{Q_{x}}m_{k}(u,D)f(z)\,dz\bigg|^{2}\bigg)^{1/2}\,dy,$$
		where $Q_{x}$ is a dyadic cube satisfying $x\in Q_{x}$.
		Then we have
		$$\|S_{2}f\|_{p}\lesssim F\|f\|_{p}.$$
	\end{lemma}
	\begin{proof}
		Define
		\begin{align*}
			\vec{T}(\{f_{k}\})(x)=\fint_{Q_{x}} \sup_{u\in I}\bigg(\sum_{|k+r_{x}|> N}\bigg|m_{k}(u,D)f_{k}(y)-\fint_{Q_{x}}m_{k}(u,D)f_{k}(z)\,dz\bigg|^{2}\bigg)^{1/2}\,dy.
		\end{align*}
		By Lemma \ref{d1}, it suffices to show
		\begin{align}\label{a33}
			\Big\|\vec{T}(\{f_{k}\})\Big\|_{2}\lesssim F\bigg\|\bigg(\sum_{k}|f_{k}|^{2}\bigg)^{1/2}\bigg\|_{2}
		\end{align}
and
		\begin{align}\label{a34}
			\Big\|\vec{T}(\{f_{k}\})\Big\|_{\infty}\lesssim F \bigg\|\bigg(\sum_{k}|f_{k}|^{2}\bigg)^{1/2}\bigg\|_{\infty}.
		\end{align}

	We start with the proof of \eqref{a33}. Similarly as \eqref{b10}, we have
		$$\Big\|\vec{T}(\{f_{k}\})\Big\|_{2}\lesssim \bigg(\sum_{k}\bigg\|\sup_{u\in I}\big|m_{k}(u,D)f_{k}\big|\bigg\|_{2}^{2}\bigg)^{1/2}\leq F\bigg\|\bigg(\sum_{k}|f_{k}|^{2}\bigg)^{1/2}\bigg\|_{2}.$$

		It remains to show \eqref{a34}.
		Notice that $\vec{T}(\chi_{Q_{x}^{*}}f_{k})(x)$ can be written as
		$$\fint_{Q_{x}} \sup_{u\in I}\bigg(\sum_{|k+r_{x}|> N}\bigg|m_{k}(u,D)(\chi_{Q_{x}^{*}}f_{k})(y)-\fint_{Q_{x}}m_{k}(u,D)(\chi_{Q_{x}^{*}}f_{k})(z)\,dz\bigg|^{2}\bigg)^{1/2}\,dy,$$
and $\vec{T}(\chi_{(Q_{x}^{*})^{\complement}}f_{k})(x)$ can be written as
		$$\fint_{Q_{x}} \sup_{u\in I}\bigg(\sum_{|k+r_{x}|> N}\bigg|m_{k}(u,D)(\chi_{(Q_{x}^{*})^{\complement}}f_{k})(y)-\fint_{Q_{x}}m_{k}(u,D)(\chi_{(Q_{x}^{*})^{\complement}}f_{k})(z)\,dz\bigg|^{2}\bigg)^{1/2}\,dy.$$	
By the triangle inequality, it suffices to show
		\begin{align}\label{a17}
			\sup_{x}\vec{T}(\{\chi_{Q_{x}^{*}}f_{k}\})(x)\lesssim F \bigg\|\bigg(\sum_{k}|f_{k}|^{2}\bigg)^{1/2}\bigg\|_{\infty}
		\end{align}
and
		\begin{align}\label{a18}
			\sup_{x}\vec{T}(\{\chi_{(Q_{x}^{*})^{\complement}}f_{k}\})(x)\lesssim F \bigg\|\bigg(\sum_{k}|f_{k}|^{2}\bigg)^{1/2}\bigg\|_{\infty}.
		\end{align}

		We first show \eqref{a17}. From the triangle inequality and the Cauchy-Schwartz inequality and Fubini's theorem, it follows that
		\begin{align*}
			\vec{T}(\{\chi_{Q_{x}^{*}}f_{k}\})(x)&\lesssim |Q_{x}|^{-1/2}\bigg( \sum_{k}\bigg\|\sup_{u\in I}\big|m_{k}(u,D)(\chi_{Q_{x}^{*}}f_{k})(\cdot)\big|\bigg\|_{2}^{2}\bigg)^{1/2}\\
			&\lesssim F|Q_{x}|^{-1/2}\bigg( \int_{Q_{x}^{*}}\sum_{k}|f_{k}(y)|^{2}\,dy\bigg)^{1/2}
			\lesssim F\sup_{y}\bigg( \sum_{k}|f_{k}(y)|^{2}\bigg)^{1/2}.
		\end{align*}

		We next show \eqref{a18}. By Lemma \ref{a20}, for all $y,z\in Q_{x}$ and $u\in I,$ we have
		\begin{align*} &\big|m_{k}(u,D)(\chi_{(Q_{x}^{*})^{\complement}}f_{k})(y)-m_{k}(u,D)(\chi_{(Q_{x}^{*})^{\complement}}f_{k})(z)\big|\\
			&\lesssim E_{k}^{u}(x,y,z)\sup_{w}|f_{k}(w)|\lesssim G C_{0,U}^{l} 2^{-\delta|k+r_{x}|}\bigg\|\bigg(\sum_{k}|f_{k}|^{2}\bigg)^{1/2}\bigg\|_{\infty}.
		\end{align*}
		As a result, we find that, for all $y\in Q_{x}$ and $u\in I,$
		\begin{align*}&\bigg(\sum_{|k+r_{x}|> N}\bigg|m_{k}(u,D)(\chi_{(Q_{x}^{*})^{\complement}}f_{k})(y)-\fint_{Q_{x}}m_{k}(u,D)(\chi_{(Q_{x}^{*})^{\complement}}f_{k})(z)\,dz\bigg|^{2}\bigg)^{1/2}\\
		&\lesssim C_{0,U}^{l}G 2^{-\delta N}\bigg\|\bigg(\sum_{k}|f_{k}|^{2}\bigg)^{1/2}\bigg\|_{\infty},\end{align*}
which in turn implies
		$$\sup_{x}\vec{T}(\chi_{(Q_{x}^{*})^{\complement}}f_{k})(x)\lesssim C_{0,U}^{l}G 2^{-\delta N}\bigg\|\bigg(\sum_{k}|f_{k}|^{2}\bigg)^{1/2}\bigg\|_{\infty}.$$
		Choose $N=\frac{1}{\delta}\log_{2}(C_{0,U}^{l}G/F)$ such that $C_{0,U}^{l}G2^{-\delta N}=F$, where $\delta$ is the one in \eqref{f2}. Then we conclude that
		$$\sup_{x}\vec{T}(\chi_{(Q_{x}^{*})^{\complement}}f_{k})(x)\lesssim F \bigg\|\bigg(\sum_{k}|f_{k}|^{2}\bigg)^{1/2}\bigg\|_{\infty}.$$
This finishes the proof of of Lemma \ref{a19}.
\end{proof}

	Now we use Lemmas \ref{a26} and \ref{a19} to prove our key Lemma \ref{a13}.
	\begin{lemma}\label{a13}
	 Suppose that there exist constants $F$ and $G$ such that $a_{k}$ satisfies
		$$\bigg\|\sup_{u\in I}\big|a_{k}(u,D)f\big|\bigg\|_{q}\leq F\|f\|_{q}\ \ \textrm{and} \ \ \big|\partial_{\xi}^{\alpha}a_{k}(u,\xi)\big|\leq G$$
 for $q\in\{2,\infty\}$, $u\in I,$ $\xi\in K_{1}\times K_{2}$, $|\alpha|\leq 6$ and $k\in \mathbb{Z}$. Then we have
		$$\bigg\|\sup_{u\in I}\bigg(\sum_{k}\big|m_{k}(u,D)f\big|^{2}\bigg)^{1/2}\bigg\|_{p}\lesssim \bigg(1+\big(\log_{2}(C_{0,U}^{l}G/F)\big)^{1/2-1/p}\bigg)F\|f\|_{p}.$$
	\end{lemma}
	\begin{proof}
		By scaling, it suffices to show
		$$\bigg\|\sup_{u\in I}\bigg(\sum_{k}\big|m_{k}(u,D_{1},2^{-v}D_{2})f\big|^{2}\bigg)^{1/2}\bigg\|_{p}\lesssim \bigg(1+\big(\log_{2}(C_{0,U}^{l}G/F)\big)^{1/2-1/p}\bigg)F\|f\|_{p}.$$
		By the fact that
		$$\bigg\|\sup_{u\in I}\bigg(\sum_{k}\big|m_{k}(u,D_{1},2^{-v}D_{2})f\big|^{2}\bigg)^{1/2}\bigg\|_{p}=\bigg\|\sup_{u\in I}\bigg(\sum_{k}\big|m_{k-v}(u,D_{1},2^{-v}D_{2})f\big|^{2}\bigg)^{1/2}\bigg\|_{p}, $$
		we should only to show
		$$\bigg\|\sup_{u\in I}\bigg(\sum_{k}\big|m_{k-v}(u,D_{1},2^{-v}D_{2})f\big|^{2}\bigg)^{1/2}\bigg\|_{p}\lesssim \bigg(1+\big(\log_{2}(C_{0,U}^{l}G/F)\big)^{1/2-1/p}\bigg)F\|f\|_{p}.$$
		
		Therefore, without loss of generality, we can assume $v=0$ and show
		$$\bigg\|\sup_{u\in I}\bigg(\sum_{k}\big|m_{k}(u,D)f\big|^{2}\bigg)^{1/2}\bigg\|_{p}\lesssim \bigg(1+\big(\log_{2}(C_{0,U}^{l}G/F)\big)^{1/2-1/p}\bigg)F\|f\|_{p}.$$
		By Fatou's lemma, it suffices to show, for all $M\geq1,$
		$$\bigg\|\sup_{u\in I}\bigg(\sum_{|k|\leq M}\big|m_{k}(u,D)f\big|^{2}\bigg)^{1/2}\bigg\|_{p}\lesssim \bigg(1+\big(\log_{2}(C_{0,U}^{l}G/F)\big)^{1/2-1/p}\bigg)F\|f\|_{p},$$
		where the implicit constant is independent of $M$.
		
		By density we can assume $f\in \mathcal{S}$. Define $$\vec{T}_{M}f=\{m_{k}(u,D)f\}_{|k|\leq M}\in L^{p}(\mathbb{R}^{2},L^{\infty}l^{2}_{M})\bigcap C(\mathbb{R}^{2},L^{\infty}l^{2}_{M}),$$
By Lemma \ref{a14}, it suffices to show
		$$\Big\|M^{\sharp}_{\textrm{dyad}}\vec{T}_{M}f\Big\|_{p}\lesssim \bigg(1+\big(\log_{2}(C_{0,U}^{l}G/F)\big)^{1/2-1/p}\bigg)F\|f\|_{p}.$$
		Choose $N=\frac{1}{\delta}\log_{2}(C_{0,U}^{l}G/F)$ and define
		$$S_{1}f(x)=\fint_{Q_{x}} \sup_{u\in I}\bigg(\sum_{|k+r_{x}|\leq N}\bigg|m_{k}(u,D)f(y)-\fint_{Q_{x}}m_{k}(u,D)f(z)\,dz\bigg|^{2}\bigg)^{1/2}\,dy$$
and
		$$S_{2}f(x)=\fint_{Q_{x}} \sup_{u\in I}\bigg(\sum_{|k+r_{x}|> N}\bigg|m_{k}(u,D)f(y)-\fint_{Q_{x}}m_{k}(u,D)f(z)\,dz\bigg|^{2}\bigg)^{1/2}\,dy.$$
		 For each $x\in \mathbb{R}^{2}$, we can find a dyadic cube $x\in Q_{x}$ such that
		$M^{\sharp}\vec{T}_{M}f(x)\leq 2S_{1}f(x)+2S_{2}f(x).$
		It suffices to obtain
		$$\|S_{1}f\|_{p}\lesssim \big(\log_{2}(C_{0,U}^{l}G/F)\big)^{1/2-1/p}F\|f\|_{p} \ \ \textrm{and} \ \ \|S_{2}f\|_{p}\lesssim F\|f\|_{p},$$
		which follows from Lemmas \ref{a26} and \ref{a19}.
	\end{proof}

	\section{the proof of the main theorem}
	This section is devoted to the proof of our main Theorem \ref{curve gamma}. First of all, we will apply Minkowski's inequality and replace the sharp truncation $[-1, 1]$ in the definition of $H^{\gamma}$ by a smooth truncation. Let $\epsilon_0>0$ be a small number that is to be determined, its choice depends only on $\gamma$, and let $\widetilde{\varphi}:\ \R\to \R$ be a non-negative smooth bump function supported on $(-\epsilon_0, \epsilon_0)$ with $\widetilde{\varphi}(t)=1$ for every $|t|\le \epsilon_0/2$. Let $x=(x_{1},x_{2})$, by Minkowski's inequality, we observe that
	\begin{equation}
	    \bigg\|\int_{-1}^1 f(x_1-t, x_2-u(x)\gamma(t))(1-\widetilde{\varphi}(t))\,\frac{dt}{t}\bigg\|_p \lesim_{\epsilon_0} \|f\|_p
	\end{equation}
	for every $p>1$. Therefore, to prove the $L^p$-boundedness of $H^{\gamma}$, it suffices to prove the same boundedness of the operator
	\begin{equation}\label{201205e6-2}
	    \int_{\R} f(x_1-t, x_2-u(x)\gamma(t))\widetilde{\varphi}(t)\,\frac{dt}{t},
	\end{equation}
	which, for the sake of simplifying notation, will still be called as $H^{\gamma}$. Without loss of generality, we assume $u>0$.

Let $\varphi\in C_{c}^{\infty}(\mathbb{R})$ be a nonnegative function, $\supp\varphi\subset[-2,2]$ and $\varphi=1$ on $[-1,1]$. Define $\phi(t)=\varphi(t)-\varphi(2t)$, then
	\begin{align}\label{f3}
		\textrm{$\supp\phi\subset \left\{t\in \mathbb{R} :\ 1/2\leq|t|\leq2\right\},\ $  $\sum_{k\in \mathbb{Z}}\phi_{k}(t)=1$ \ \ for all $t\neq0$,}
	\end{align}
	where $\phi_{k}(t)=\phi(2^{-k}t)$. Let $\tilde{\phi}\in C_{c}^{\infty}(\mathbb{R}\backslash 0)$ such that $\tilde{\phi}=1$ on $\{t\in \mathbb{R} :\ 1/4\leq|t|\leq4\}$. We then define $P_{k}^{2}f(x)=(\phi_{k}(\xi_{2})\widehat{f}(\xi))^\vee(x)$ and $\tilde{P}_{k}^{2}f(x)=(\tilde{\phi}_{k}(\xi_{2})\widehat{f}(\xi))^\vee(x)$ for $k\in \mathbb{Z}$.
	
	To proceed, we will reduce the $L^p$-boundedness of $H^{\gamma}$ to a square function estimate.
	Recall $u:\ \R^2\to \R$ is a Lipschitz function with Lipschitz norm $\|u\|_{\lip}\le 1/2$. This further implies that  $\|u(x_1,\cdot)\|_{\lip}\leq 1/2\ $ for every $x_1\in \mathbb{R}$.
	\begin{lemma}\label{a1}
		For every $1<p<\infty$, we have
		\begin{align*}
	\bigg\|\sum_{k}(1-\tptwo_k)H^{\gamma}P_{k}^{2}f\bigg\|_{p}\lesssim \|f\|_{p}
		\end{align*}
		for every Schwartz function $f$.
	\end{lemma}
	\begin{proof}[Proof (of Lemma \ref{a1})]
		This lemma follows from Theorem 1.1 in \cite{MR3841536}, which further relies on Jones' beta numbers \cite{MR1013815}. In the special case where $\gamma$ is a monomial, such a lemma was already proved and used in \cite{MR3841536}.  The proof of the general case is essentially the same. For the sake of completeness, we still include the proof here.
		
		We first write
		\begin{equation}
		   (1-\tptwo_k) H^{\gamma} \ptwo_k f(x)=\int_{\R} (1-\tptwo_k) \big[\ptwo_k f(x_1-t, x_2-u(x)\gamma(t))\big]\widetilde{\varphi}(t)\,\frac{dt}{t}.
		\end{equation}
		By Minkowski's inequality, we have
		\begin{equation}
		   \Big\|(1-\tptwo_k) H^{\gamma} \ptwo_k f\Big\|_{p} \le \int_{\R}\Big \|(1-\tptwo_k) \big[\ptwo_k f(x_1-t, x_2-u(x)\gamma(t))\big]\Big\|_{p} \widetilde{\varphi}(t)\,\frac{dt}{|t|}.
		\end{equation}
		By \cite[Theorem 1.1]{MR3841536}, the last display can be further bounded by
		\begin{equation}\label{201127e6-5}
		     \int_{\R} \Big\|\|f(x_1-t, x_2)\|_{L^p_{x_2}}\|u(x_1, \cdot) \gamma(t)\|_{\lip}\Big\|_{L^p_{x_1}} \widetilde{\varphi}(t)\,\frac{dt}{|t|}.
		\end{equation}
Note that $\|u(x_1, \cdot) \gamma(t)\|_{\lip}\le |\gamma(t)|\|u\|_{\lip}$. Therefore, \eqref{201127e6-5} is bounded by
\begin{equation}
    \|f\|_p \|u\|_{\lip} \int_{\R} |\gamma(t)|\widetilde{\varphi}(t)\,\frac{dt}{|t|}\lesssim \|f\|_p \|u\|_{\lip}.
\end{equation}
In the last step, we used the fact that $|\gamma(t)|\lesssim |t|^{\log_2 \czerol}$, which is a simple consequence of \eqref{curve gamma_z}.
	\end{proof}

Lemma \ref{a1} tells us that in order to prove the $L^p$-boundedness of $H^{\gamma}$, it suffices to consider the $L^p$-boundedness of the ``diagonal" term:
\begin{equation}
    \bigg\|\sum_k \tptwo_k H^{\gamma} \ptwo_k f\bigg\|_p \lesssim \|f\|_p
\end{equation}
for every $p>1$. By the Littlewood-Paley theory and the vector-valued estimates for the maximal operator, it suffices to prove
\begin{equation}\label{201127e6-8}
    \bigg\|\bigg(\sum_k  |H^{\gamma} \ptwo_k f|^2\bigg)^{1/2}\bigg\|_p \lesssim \|f\|_p.
\end{equation}
The proof of this square function estimate will occupy the rest of the paper. In order to bound such a square function, we will first apply a dyadic decomposition in the $t$ variable. For every $u\in \R$, define
\begin{equation}\label{201127e6-9}
    A_{u,l}f(x)=\int_{\R} f(x_{1}-t,x_{2}-u\gamma(t))\phi(2^{l}t)\widetilde{\varphi}(t)\,\frac{dt}{t}.
\end{equation}
We use notations $u(x)=u_{x}$. With these notations, we can write $H^{\gamma}\ptwo_k f$ as a sum of $A_{u_{x}, l}\ptwo_k f$ over $l\in \N$. For each fixed $k\in \Z$, we will split the sum over $l\in \N$ into two parts: A high frequency part and a low frequency part. The intuition behind the splitting is quite standard: If $l$ is ``large", that is, $t\simeq 2^{-l}$ is ``small", then by the uncertainty principle applied to the second variable, the term $u \gamma$ can be treated as a small perturbation since $\ptwo_k f$ is essentially constant at the scale $2^{-k}$ in the second variable; if $l$ is ``small", that is, $t$ is ``large", then we will be able to make use of oscillations to obtain certain exponential decay in $l$, via the local smoothing estimates that were established in the following Lemma \ref{201127lem6-4}. \\

Let us proceed with the decomposition in $l$ described above. For $v\in \Z$, define
\begin{equation}
    E_{v}=\{k\in \Z:\ 2^{v+k}\gamma(\epsilon_{0}/2)\geq1\}.
\end{equation}
Define $v_{x}\in \mathbb{Z}$ satisfies $2^{-v_{x}}u_{x}\in[1,2)$.
We see that if $k\notin E_{v_x}$, then
\begin{equation}
    2^{v_x+k}\gamma(\epsilon_{0}/2)\simeq |u(x)|2^k \gamma(\epsilon_{0}/2)\lesim 1,
\end{equation}
and as a consequence, for every choice of $l$, the term $u\gamma$ can be treated as a perturbation when considering $\ptwo_k f$. To be more precise, we have
\begin{lemma}\label{201128lem6-2}
    For every $x\in \R^2$ and $k\notin E_{v_x}$, we have the pointwise estimate
    \begin{equation}\label{b5}
			\big|H^{\gamma}P_{k}^{2}f(x)\big|\lesssim H_1^{*}P_{k}^{2}f(x)+ M_{S}P_{k}^{2}f(x).
    \end{equation}
    Here and hereafter, $H_1^*$ denotes the maximal Hilbert transform applied in the first variable, $M_{S}$ denotes the composition of some $M_{1}$ and $M_{2}$, which are the maximal function applied in the first and second variables, respectively.
\end{lemma}
Next, for $v, k\in \Z$, define
\begin{equation}\label{201127e6-13}
    l_0=l_0(v, k)\in \mathbb{Z} \ \text{ such that }\ 2^{v+k} \gamma(2^{-l_0})\simeq 1.
\end{equation}
There may be multiple choices of $l_0$, we here pick an arbitrary one. Similar to Lemma \ref{201128lem6-2}, we have the following pointwise estimate, i.e., Lemma \ref{201128lem6-3}. All of these two lemmas will be proved in Section $6$.
\begin{lemma}\label{201128lem6-3}
    For every $x\in \R^2$ and $k\in E_{v_x}$, it holds that
    \begin{equation}
        \bigg|\sum_{l\ge l_0(v_x, k)} A_{u_x, l} \ptwo_k f(x)\bigg|\lesim H_1^* \ptwo_k f(x)+ M_S \ptwo_k f(x).
    \end{equation}
\end{lemma}
By Lemmas \ref{201128lem6-2} and \ref{201128lem6-3} and the vector-valued estimates for $H_1^*$ and $M_S$, to prove \eqref{201127e6-8}, it remains to prove
\begin{equation}\label{201127e6-15}
    \bigg\|\bigg( \sum_k \chi_{\{k\in E_{v_x}\}} \bigg|\sum_{l\le l_0(v_x, k)} A_{u_x, l} \ptwo_k f(x)\bigg|^2 \bigg)^{1/2} \bigg\|_{p} \lesim \|f\|_p
\end{equation}
for every $p>1$. For $k\in \mathbb{Z}$, define $P_{k}^{1}f(x)=(\phi_{k}(\xi_{1})\widehat{f}(\xi))^\vee(x)$.
Now we carry out a last decomposition, which is a Littlewood-Paley decomposition in the first variable of $f$, write
\begin{equation}\label{201127e6-16}
    A_{u_x, l}\ptwo_k f(x)=\sum_{k'\in \Z} A_{u_x, l}\pone_{k'} \ptwo_k f(x).
\end{equation}
For each fixed $u$, the Fourier multiplier of the convolution operator $A_{u, l}$ is given by
\begin{equation}\label{201127e6-17}
    \int_{\R} e^{i(t\xi+u\gamma(t)\eta)} \phi(2^{l}t)\widetilde{\varphi}(t)\,\frac{dt}{t}=\int_{\R} e^{i(2^{-l} t\xi+u\gamma(2^{-l} t)\eta)} \phi(t)\widetilde{\varphi}(2^{-l}t)\,\frac{dt}{t}.
\end{equation}
Assume that $|\xi|\simeq 2^{k'}$ and $|\eta|\simeq 2^k$. In order for the multiplier to have ``non-trivial" contributions, it should admit a critical point in the region where $|t|\simeq 2^{-l}$, that is,
\begin{equation}
    \xi+u\gamma'(t)\eta=0\  \text{ for some }\ |t|\simeq 2^{-l}.
\end{equation}
This leads to
\begin{equation}\label{201127e6-19}
    2^{k'}\simeq |u| \gamma'(2^{-l}) 2^k.
\end{equation}
In other words, for fixed $u, k$ and $l$, in the sum over $k'$ in \eqref{201127e6-16}, only those terms satisfying \eqref{201127e6-19} will have non-trivial contributions; all other terms can be handled via elementary methods like integration by parts.  Therefore, to prove \eqref{201127e6-15}, we will only prove
\begin{equation}\label{201127e6-20}
    \bigg\|\bigg( \sum_k \chi_{\{k\in E_{v_x}\}} \bigg|\sum_{l\le l_0(v_x, k)} A_{u_x, l} \pone_{\kone}\ptwo_k f(x)\bigg|^2 \bigg)^{1/2} \bigg\|_{p} \lesim \|f\|_p,
\end{equation}
where
\begin{equation}
    \kone=\kone(v_x, k, l)\in \mathbb{Z}\ \ \text{is such that}\ \ 2^{\kone}\simeq 2^{v_x} 2^k \gamma'(2^{-l}).
\end{equation}
In the following part, for ease of notation, we will compress the dependence of $\kone$ on $x, k$ and $l$. We will emphasize such dependence whenever necessary. Moreover, if $x$ and $k$ are clear from the context, we will also leave out the dependence of $l_0$ on $x$ and $k$. \\

One key step in the proof of \eqref{201127e6-19} is the following estimate of local smoothing type, which originally only holds for $p>2$. Here we have a same estimate for every $p>1$ due to the Lipschitz assumption on $u$.
\begin{lemma}\label{201127lem6-4}
    For every $p>1$, it holds that
    \begin{equation}
        \Big\|\chi_{\{k\in E_{v_x}\}} A_{u_x, l_0-l}\pone_{\kone} \ptwo_k f(x) \Big\|_{p} \lesim 2^{-\delta_p l} \|f\|_p
    \end{equation}
    for every $l\in \N$, where $\delta_p>0$ is a small constant, $\kone=\kone(v_x, k, l_0-l)$ and $l_0=l_0(v_x, k)$.
\end{lemma}

Note that by Fubini's theorem, the desired estimate \eqref{201127e6-20} at $p=2$ follows immediately from Lemma \ref{201127lem6-4}.

The proof of Lemma \ref{201127lem6-4} will be given in Section $7$, we now make one remark. Recall that the phase function of the multiplier in \eqref{201127e6-17}, given  by
\begin{equation}\label{201127e6-23}
    2^{-(l_0-l)} t\xi+u_x  \gamma(2^{-(l_0-l)}t)\eta.
\end{equation}
Note that for every $t$ in the support of $\phi$, it holds that $|t|\simeq 1$; for every $(\xi, \eta)$ in the Fourier support of $\pone_{\kone} \ptwo_k f$, it holds that
\begin{equation}
    |\xi|\simeq 2^{\kone}\simeq |u_x| 2^k \gamma'(2^{-(l_0-l)})\ \ \textrm{and}\ \ |\eta|\simeq 2^k.
\end{equation}
Let us show that both of the terms in \eqref{201127e6-23} have scales bigger than one. This is another way of saying that our phase function \eqref{201127e6-23} may oscillate very fast; local smoothing estimates of Mockenhaupt, Seeger and Sogge \cite{MR1168960} allow us to explore this high oscillation and obtain a favourable decay estimate as in the above lemma. That both terms in \eqref{201127e6-23} are of scale bigger than one is trivial in the case when $\gamma$ is a monomial. In the general case, the assumptions we made on $\gamma$ in Theorem \ref{f19} will appear. Let us be more precise. First, note that
\begin{equation}
    |u_x| \gamma(2^{-l_0+l}t)|\eta|\ge |u_x| \gamma(2^{-l_0})|\eta|\gtrsim 1,
\end{equation}
which follows from the monotonicity assumption on $\gamma$ and the definition of $l_0$ in \eqref{201127e6-13}. The other term is the more interesting one. We need to show that
\begin{equation}
    2^{-l_0+l} |u_x| 2^k \gamma'(2^{-l_0+l})\gtrsim 1,
\end{equation}
which follows from applying item (ii) in the statement of Theorem \ref{f19} with $j=1$. \\

So far we have finished the proof of \eqref{201127e6-20} at $p=2$, which leads to the main Theorem \ref{curve gamma} at $p=2$. Next, we consider $p\neq 2$. It turns out that the case $p<2$ and the case $p>2$ require entirely different tools. The case $p<2$ is easier, and can be handled essentially via the interpolation and the bootstrapping argument of Nagel, Stein and Wainger \cite{MR0466470}. The case $p>2$ is much more complicated, and would require a vector-valued version of the $L^p$ orthogonality principle of Seeger \cite{MR955772}, we here will prove this vector-valued result in a general homogeneous type space instead of $\mathbb{R}^{n}$. Homogeneous spaces were introduced
by Coifman and Weiss \cite{MR0499948}, who discovered that the ``spaces of
homogeneous type'' are the metric spaces to which the Calder$\acute{\textrm{o}}$n-Zygmund
theory extends naturally. It would be interesting to find an approach that works for every $p$ simultaneously. \\

Let us first look at the easier case $p<2$. By interpolation with the exponential decay estimate in Lemma \ref{201127lem6-4} and the Littlewood-Paley inequality, it suffices to prove that
\begin{equation}\label{201127e6-27}
    \bigg\|\bigg( \sum_k \chi_{\{k\in E_{v_x}\}} \bigg| A_{u_x, l_0-l} \pone_{\kone}\ptwo_k f(x)\bigg|^{2} \bigg)^{1/2} \bigg\|_{p} \lesim   \bigg\|\bigg(\sum_k |\ptwo_k f|^{2}\bigg)^{1/2}\bigg\|_p
\end{equation}
for every $l\in \N$. Here $\kone=\kone(v_x, k, l_0-l)$ and $l_0=l_0(v_x, k)$. In order to prove the above vector-valued estimates, we will appeal to the $L^p$-boundedness of the maximal operators that were established in \cite{NjSlHx20}. Indeed, one can also apply Lemma \ref{201127lem6-4} and recover the above result. Define
\begin{equation}
    M^{\gamma} f(x)=\sup_{\epsilon>0} \frac{1}{2\epsilon}\int_{-\epsilon}^{\epsilon} |f(x_1-t, x_2-u(x)\gamma(t))|\widetilde{\varphi}(t)\,dt.
\end{equation}
It was proven in \cite{NjSlHx20} that, under the same assumptions as in Theorem \ref{f19}, it holds that
\begin{equation}\label{201208e5-29}
    \|M^{\gamma} f\|_p \lesim \|f\|_p
\end{equation}
for every $p>1$. Moreover, one can upgrade such scalar-valued bound into a vector-valued bound for free. To be more precise, we also have that
\begin{equation}\label{201201e6-30}
    \bigg\|\bigg(\sum_k |M^{\gamma} f_k|^2\bigg)^{1/2}\bigg\|_p \lesim \bigg\|\bigg(\sum_k | f_k|^2\bigg)^{1/2}\bigg\|_p
\end{equation}
for every $1<p\le 2$. To show \eqref{201127e6-27}, we observe the following pointwise estimate
\begin{equation}
    \big|A_{u_x, l_0-l} \pone_{\kone}\ptwo_k f(x)\big|\lesim M^{\gamma} M_S P^2_k f(x).
\end{equation}
Therefore, the desired estimate \eqref{201127e6-27} follows immediately from \eqref{201201e6-30} and the vector-valued estimates for the strong maximal operator. \\

In the remaining part, we focus on the case $p>2$. By interpolating with the $L^2$ bounds in Lemma \ref{201127lem6-4}, it suffices to prove
\begin{equation}\label{201127e6-20zz}
    \bigg\|\bigg( \sum_k \chi_{\{k\in E_{v_x}\}} \bigg| A_{u_x, l_0-l} \pone_{\kone}\ptwo_k f(x)\bigg|^2 \bigg)^{1/2} \bigg\|_{p} \lesim (l+1)^{10}  \bigg\|\bigg(\sum_k |\ptwo_k f|^2\bigg)^{1/2}\bigg\|_p
\end{equation}
for every $l\ge 0$, where
\begin{equation}
    \kone=\kone(v_x, k, l_0-l) \ \ \text{is such that} \ \ 2^{\kone}\simeq 2^{v_x} 2^k \gamma'(2^{-(l_0-l)}).
\end{equation}
We bound the left hand side of \eqref{201127e6-20zz} by
\begin{equation}\label{201201e6-34}
    \Bigg\| \Bigg[\sum_{v\in \Z} \sup_{u\in [1, 2]} \bigg( \sum_k \chi_{\{k\in E_{v}\}} \bigg| A_{2^v u, l_0-l} \pone_{\kone}\ptwo_k f\bigg|^2 \bigg)^{p/2}\Bigg]^{1/p} \Bigg\|_{p},
\end{equation}
where
\begin{equation}\label{201201e6-35}
    l_0=l_0(v, k) \ \ \text{ satisfying }\ \ 2^{v+k}\gamma(2^{-l_0})\simeq 1,
\end{equation}
and
\begin{equation}\label{201201e6-36}
    \kone=\kone(v, k, l_0-l) \ \ \text{ satisfying } \ \ 2^{\kone}\simeq 2^v 2^k \gamma'(2^{-(l_0-l)}).
\end{equation}

We now state two lemmas.

\begin{lemma}\label{lem21-315-51}
For each fixed $k_{1}$, $k$ and $l$, there are at most $C(l+1)$ choices of $v$ such that \eqref{201201e6-36} holds.
\end{lemma}
\begin{proof}[Proof (of Lemma \ref{lem21-315-51})]
Fix $k_{1}$, $k$ and $l$, by the assumption (ii) of Theorem \ref{f19}, \eqref{201201e6-35} and \eqref{201201e6-36}, we have $$C_{0,L}^{l}2^{l_{0}-l}\lesssim2^{\kone}\simeq \gamma_{l_{0}}(2^{l})2^{l_{0}-l}\leq C_{0,U}^{l}2^{l_{0}-l}.$$
Therefore, for each fixed $k_{1}$, there are at most $C_{1}(l+1)$ choices of $l_{0}$. By \eqref{201201e6-35}, for each fixed $l_{0}$, there are at most $C_{2}$ choices of $v$. Then for each fixed $k_{1}$, there are at most $C_{1}C_{2}(l+1)$ choices of $v$ such that \eqref{201201e6-36} holds.
\end{proof}
\begin{lemma}\label{a39}
		Let $2<p<\infty$ and $v\in \Z$. Then
		\begin{align}\label{a41}
			\bigg\|\sup_{ u\in[1,2]}\bigg(\sum_{k\in E_{v}}\bigg|A_{2^{v}u,l_{0}-l}\pone_{\kone} \ptwo_{k} f\bigg|^{2}\bigg)^{1/2}\bigg\|_{p}\lesssim (l+1)^{5} \bigg\|\bigg(\sum_{k\in E_{v}}\big|\pone_{\kone}\ptwo_k f\big|^{2}\bigg)^{1/2}\bigg\|_{p},
		\end{align}
		where $l_0$ and $k_1$ are defined as in \eqref{201201e6-35} and \eqref{201201e6-36}, respectively.
	\end{lemma}
	We first assume Lemma \ref{a39}, which will be proved in Section $8$, and continue the proof of \eqref{201127e6-20zz}. By \eqref{201201e6-34}, $l^{2}\subseteq l^{p}$ and Lemmas \ref{lem21-315-51} and \ref{a39}, the left hand side of \eqref{201127e6-20zz} can be bounded by
	\begin{align}\label{21-315-538}
	    &(l+1)^{5} \Bigg\| \Bigg[\sum_v \bigg(\sum_{k\in E_{v}}\big|\pone_{\kone}\ptwo_k f\big|^{2}\bigg)^{p/2}\Bigg]^{1/p}\Bigg\|_{p}
\leq (l+1)^{5} \bigg\| \bigg(\sum_v \sum_{k\in E_{v}}\big|\pone_{\kone}\ptwo_k f\big|^{2}\bigg)^{1/2}\bigg\|_{p}\nonumber\\
&\lesssim  (l+1)^{10} \bigg\| \bigg(\sum_{k'\in \Z} \sum_{k\in \Z}\big|\pone_{k'}\ptwo_k f\big|^{2}\bigg)^{1/2}\bigg\|_{p}.
	\end{align}
In the end, we apply the bi-parameter Littlewood-Paley inequality and finish the proof of \eqref{201127e6-20zz}.\\

\section{Proofs of Lemmas \ref{201128lem6-2} and \ref{201128lem6-3}}

Let us begin with the proof of Lemma \ref{201128lem6-2}. We fix $x\in \R^2$ and $k\notin E_{v_x}$, which means $2^{v_x+k}\lesssim 1$. We have
\begin{equation}\label{201202e7-40}
    H^{\gamma} \ptwo_k f(x)=\int_{\R} \ptwo_k f(x_1-t, x_2-u_x \gamma(t))\widetilde{\varphi}(t)\,\dtt.
\end{equation}
From $\ptwo_k f=\ptwo_k \tilde{P}_{k}^{2}f$ and the definition of $\tilde{P}_{k}^{2}f$, it easily follows that
\begin{equation}
    \ptwo_k f(x)=\int_{\R}\ptwo_k f(x_1, x_2-z)2^k \widecheck{\tilde{\phi}}(2^k z)\,dz.
\end{equation}
We write \eqref{201202e7-40} as
\begin{equation}
    \begin{split}
        & \int_{\R} \int_{\R}\ptwo_k f(x_1-t, x_2-u_x \gamma(t)-z) 2^k \widecheck{\tilde{\phi}}(2^k z) \widetilde{\varphi}(t)\,dz \,\dtt\\
        & =\int_{\R} \int_{\R}\ptwo_k f(x_1-t, x_2-z) 2^k \widecheck{\tilde{\phi}}(2^k (z+u_x \gamma(t))) \widetilde{\varphi}(t)\,dz\, \dtt.
    \end{split}
\end{equation}
We compare it with
\begin{equation}
    \int_{\R} \int_{\R}\ptwo_k f(x_1-t, x_2-z) 2^k \widecheck{\tilde{\phi}}(2^k  z) \widetilde{\varphi}(t)\,dz\, \dtt,
\end{equation}
which is certainly bounded by $H_1^* \ptwo_k f$. The difference is controlled by
\begin{equation}\label{201203e7-5}
    \begin{split}
        & \int_{\R} \int_{\R} |\ptwo_k f(x_1-t, x_2-z)| 2^k \big|\widecheck{\tilde{\phi}}(2^k (z+u_x \gamma(t)))-\widecheck{\tilde{\phi}}(2^k z)\big| \widetilde{\varphi}(t) \,dz \,\frac{dt}{|t|}\\
        & \lesim \sum_{\tau\in \N} (1+\tau)^{-100} \int_{\R} \int_{\R} |\ptwo_k f(x_1-t, x_2-z)| 2^k \chi_{\{\tau 2^{-k}\le |z|\le (\tau+1) 2^{-k}\}}  \widetilde{\varphi}(t)\, dz \frac{|\gamma(t)|}{|t|}\,dt.
    \end{split}
\end{equation}
Recall that $|\gamma(t)|\lesssim |t|^{\log_2 \czerol}$, we see from it that \eqref{201203e7-5} is bounded by the strong maximal function $M_{S}P_{k}^{2}f$. This finishes the proof of Lemma \ref{201128lem6-2}.\\

Next we prove Lemma \ref{201128lem6-3}. Fix $x\in \R^2$, we abbreviate $l_0(v_x, k)$ to $l_0$, and then
\begin{equation}
    \sum_{l\ge l_0} A_{u_x, l} \ptwo_k f(x)=\sum_{l\ge l_0} \int_{\R} \ptwo_k f(x_1-t, x_2-u_x \gamma(t)) \phi_l(t)\widetilde{\varphi}(t)\,\dtt.
\end{equation}
We still compare it with
\begin{equation}
    \sum_{l\ge l_0} \int_{\R} \ptwo_k f(x_1-t, x_2) \phi_l(t)\widetilde{\varphi}(t)\,\dtt,
\end{equation}
which can be bounded by the maximal Hilbert transform $H_1^* \ptwo_k f$. The difference is bounded by
\begin{equation}\label{201203e7-8}
    \begin{split}
        & \sum_{\tau\in \N} (1+\tau)^{-100} \int_{|t|\lesim 2^{-l_0}} \int_{\R} |\ptwo_k f(x_1-t, x_2-z)| 2^k \chi_{\{\tau 2^{-k}\le |z|\le (\tau+1) 2^{-k}\}}  \widetilde{\varphi}(t) \,dz \frac{2^k |u_x||\gamma(t)|}{|t|}\,dt\\
        & \lesim \sum_{\tau\in \N} (1+\tau)^{-100} \int_{|t|\lesim 2^{-l_0}} \int_{\R} |\ptwo_k f(x_1-t, x_2-z)| 2^k \chi_{\{\tau 2^{-k}\le |z|\le (\tau+1) 2^{-k}\}}  \widetilde{\varphi}(t) \,dz \frac{|\gamma(t)|}{\gamma(2^{-l_0})|t|}\,dt.
    \end{split}
\end{equation}
Recall that from Lemma \ref{curve gamma_z}, we have $\gamma(t)\le  \gamma(2t)/\czerol$. This implies that if $|t|\simeq 2^{-l}\le 2^{-l_0}$, then
\begin{equation}
    |\gamma(t)|\le (\czerol)^{-(l-l_0)} \gamma(2^{-l_0}).
\end{equation}
From this we see again that \eqref{201203e7-8} can be bounded by the strong maximal function $M_{S}P_{k}^{2}f$. This finishes our proof.

\section{Proof of Lemma \ref{201127lem6-4}}
By interpolating with the estimate \eqref{201208e5-29}, we see that to prove Lemma \ref{201127lem6-4}, it suffices to prove
\begin{equation}\label{201205e8-40}
\Big\|\chi_{\{k\in E_{v_x}\}} A_{u_x, l_0-l}\pone_{\kone} \ptwo_k f(x) \Big\|_{p} \lesim C_{0,L}^{-\delta_{p}l} \|f\|_p
\end{equation}
for some $p\ge 6$, where $\delta_p>0$ is allowed to depend on $p$. Here $\kone=\kone(v_x, k, l_0-l)$ is such that
\begin{equation*}
    2^{k_1}\simeq 2^{v_x} 2^k 2^{l_0-l} \gamma(2^{-(l_0-l)}).
\end{equation*}
We bound the left hand side of \eqref{201205e8-40} by
\begin{align}\label{201205e8-42}
\bigg\| \sup_{v\in \mathbb{Z}}\sup_{u\in[1,2]}\chi_{\{k\in E_{v}\}}\big|A_{2^{v}u, l_0-l}\pone_{\kone} \ptwo_k f\big| \bigg\|_{p}.
\end{align}
Now we bound the sup over $v\in \Z$ by a $\ell^p$ norm. By Lemma \ref{lem21-315-51} and the method similar as in the proof of \eqref{21-315-538}, it then suffices to prove
\begin{align}\label{201205e8-43}
\bigg\| \sup_{u\in[1,2]}\chi_{\{k\in E_{v}\}}\big|A_{2^{v}u, l_0-l}\pone_{\kone} \ptwo_k f\big| \bigg\|_{p}\lesssim C_{0,L}^{-\delta_{p}l}\|f\|_p
\end{align}
uniformly in $v\in \Z$, where $\kone=\kone(v, k, l_0-l)$.
By scaling and the support assumption in \eqref{201205e6-2}, it suffices to show
\begin{align}\label{201205e8-44}
\bigg\| \sup_{u\in[1,2]}\Big|m_{\log_{2}\gamma_{l_0}(2^{l})}^{l_{0}-l}(u,D)f\Big| \bigg\|_{p}\lesssim \gamma_{l_0}(2^{l})^{-\delta_{p}}\|f\|_p\lesssim C_{0,L}^{-\delta_{p}l}\|f\|_p
\end{align}
 for all $k,v\in \mathbb{Z}$ satisfying $k\in E_{v}$ and $2^{-l_{0}+l}\lesssim \epsilon_{0}$, where
\begin{equation}\label{201211e7.6}
m_{j}^{l}(u,\xi)=\phi_{j}(\xi_{1})\phi_{j}(\xi_{2}) \int_{\R} e^{i\xi_{1}t}e^{i\xi_{2}u\gamma_{l}(t)}\phi(t)\,\dtt .
\end{equation}
We remark that $m_j^l(u, \xi)$ is the multiplier of the convolution operator $A^l_u \pone_j \ptwo_j$, where
\begin{equation}\label{201210e7.7}
A^l_u f(x)=\int_{\R} f(x_1-t, x_2-u\gamma_l(t))\phi(t)\,\dtt.
\end{equation}
Let us organize what we need to prove in the following lemma.
\begin{lemma}\label{201210lem7.1}
For every $j\ge 1$ and $l\in \N$ with $2^l\gg 1$, it holds that
\begin{equation}
\bigg\|\sup_{u\in [1, 2]} \big|m_j^l(u, D) f\big|\bigg\|_{p}\lesssim 2^{-\delta_p j} \|f\|_p
\end{equation}
for some $p\ge 6$ and some $\delta_p>0$ that is allowed to depend on $p$. In particular, the implicit constant is uniform in $l$.
\end{lemma}

The rest of this section is devoted to the proof of Lemma \ref{201210lem7.1}. Before we start the proof, let us make a remark that this is the place where we need to apply Lemma \ref{b44}, which states that $\gamma_l$ behaves ``uniformly" in the parameter $l$. In particular, if we assume that $\gamma:\ \R\to \R$ is homogeneous, then $\gamma_l$ will be independent of $l$, and certainly all estimates we obtained will also be uniform in $l$. \\

%
Let us begin the proof with an estimate for a fixed $u$, which is a simple corollary of van der Corput's lemma.
\begin{lemma}\label{b28}
For every $j\ge 1$, $2^{l}\gg 1$, $u\in [1, 2]$ and every $p\ge 2$, it holds that
		\begin{align}\label{b27}
			\Big\|m_j^l(u, D) f\Big\|_{p}\lesssim 2^{-j/p} \|f\|_p,
		\end{align}
where the implicit constant is uniform in $l$.
\end{lemma}
		\begin{proof}[Proof (of Lemma \ref{b28})]
By interpolation, it suffices to prove \eqref{b27} at $p=2$ and $p=\infty$. The estimate at $p=\infty$ is trivial since $A^l_u$ is an averaging operator. The desired $L^2$ bound follows immediately from the following pointwise estimate of the multiplier:
			\begin{align}\label{b23}
				|m_{j}^{l}(u,\xi)|\lesssim 2^{-j/2},
			\end{align}
which further follows from van der Corput's lemma and
\begin{equation}
|\gamma_{l}''(t)|\gtrsim 1 \ \text{ for all}\ t\in [1/2, 2],
\end{equation}
uniformly in $l$, as stated in Lemma \ref{b44}.
	\end{proof}

Let us continue with the proof of Lemma \ref{201210lem7.1}. In the next lemma, we prove a local smoothing estimate, which, combined with Lemma \ref{b28} and the Sobolev embedding inequality, further implies Lemma \ref{201210lem7.1}. We refer to \cite{NjSlHx20} for a more detailed discussion.

%

	\begin{lemma}\label{b56}
For every $j\ge 1$, $2^l\gg 1$ and every $p\ge 6$, it holds that
		\begin{align}\label{b29}
			\bigg(\int_{1}^{2}\Big\|m_j^l(u, D)f\Big\|_{p}^{p}\,du\bigg)^{1/p}\lesssim 2^{-(1/p+\delta_{p})j}\|f\|_{p},
		\end{align}
for some $\delta_p>0$, uniformly in $l$.
	\end{lemma}
	\begin{proof}[Proof (of Lemma \ref{b56})] Recall the definition of the multiplier $m_j^l$ in \eqref{201211e7.6}. Let $t_c$ be the critical point of the phase function there, that is,
\beq\label{21-316-712}
\xi_1+ \xi_2 u \gamma'_l(t_c)=0.
\endeq
Next we show if $t_{c}$ satisfies \eqref{21-316-712}, then $t_{c}\simeq1$.
By \eqref{21-316-712}, for all $\xi_{1},\xi_{2}\in\supp \phi_{j}$ and $u\in (1/2,4)$, one has
\begin{align}\label{21-316-713}
\gamma'_l(t_c)\in [2^{-4},2^{4}].
\end{align}
By the doubling property of $\gamma'$ in Lemma \ref{c1}, there exits a constant $c_{3}$(independent of $l$) such that if $c_{3}\leq l$, we have
\begin{align}\label{21-316-714}
	[2^{-4},2^{4}]\subset\gamma_{l}'([2^{-c_{3}},2^{c_{3}}]).
\end{align}
We now define $\epsilon_{0}= 2^{-c_{3}-4}$. The condition $c_{3}\leq l$ is only to make sure $\gamma_{l}'(t)$ is well defined on $[2^{-c_{3}},2^{c_{3}}]$ and will be satisfied by the support condition in \eqref{201205e6-2}. By \eqref{21-316-713} and \eqref{21-316-714}, we have $t_{c}\in[2^{-c_{3}},2^{c_{3}}]$.
By the stationary phase formula in \cite{MR1232192}, we can write
\beq
m_j^l(u, \xi)=a_{j,l}(u, \xi) e^{i\Psi_l(u, \xi)} +e_{j,l}(u, \xi),
\endeq
where
\beq
\Psi_l(u, \xi)=\xi_1 t_c(\xi)+ \xi_2 u \gamma_l(t_c(\xi)),
\endeq
$a_{j,l}(u, \xi)$ is a symbol of order $-1/2$ that is supported on $\{\xi\in \mathbb{R}^2 :\  |\xi_1|\simeq |\xi_2|\simeq 2^j\}$ and $e_{j,l}$ is a smooth symbol. The assumption (iii) in Theorem \ref{f19} guarantees that everything is uniform in $l$ whenever $2^l\gg 1$. It suffices to show that
\beq\label{21-315-716}
\bigg\|\int_{\R^2} \widehat{f}(\xi) a_{j,l}(u, \xi) e^{i\Psi_l(u, \xi)} e^{ix\cdot \xi} \,d\xi\bigg\|_{L^p([1, 2]\times \R^2)}  \lesssim 2^{-(1/p+\delta_p)j}\|f\|_p.
\endeq
We will show that $\Psi_l(u, \xi)$ satisfies Sogge's cinematic curvature condition in \cite{MR4078231}. Define $Q$ is the orthogonal matrix such that $Qe_{1}=(\frac{\sqrt{2}}{2},-\frac{\sqrt{2}}{2})$, $\Psi_l^{Q}(x,u,\xi)=\Psi_l(x,u,Q\xi)$ and $a_{j,l}^{Q}(x,u,\xi)=a_{j,l}(x,u,Q\xi)$. On $\supp a_{j,l}^{Q}$, by the assumption (i) of Theorem \ref{f19}, we have
\begin{align}\label{9}
\big|\partial_{\xi_{2}}^{2}\partial_{u}\Psi_l^{Q}(x,u,\xi)\big|\gtrsim1
\end{align}
uniformly in $l$ whenever $2^l\gg 1$. So we have verified the cinematic curvature condition in the first quadrant in $\xi$-space. For other quadrants, the proofs are similar and we leave them out. With Sogge's cinematic curvature condition of phase function $\Psi_l(u, \xi)$ at our disposal, \eqref{21-315-716} follows by the local smoothing estimate proven in \cite{MR4078231}. One can read \cite{NjSlHx20} for more details.
\end{proof}

\section{Proof of Lemma \ref{a39}}
		
		We first show that, for all $v\in \mathbb{Z}$ and $l\geq0$, we have
		\begin{align}\label{a40}
			\bigg\|\sup_{ u\in[1,2]}\bigg(\sum_{k\in E_{v}}\bigg|A_{2^{v}u,l_{0}-l}P_{\kone}^{1}P_{k}^{2}f\bigg|^{2}\bigg)^{1/2}\bigg\|_{p}\lesssim (l+1)^{5}\|f\|_{p}.
		\end{align}
		Define
		\begin{align*}
			m_{k}^{l}(u,\xi)=\int_{\R} e^{i2^{-l_{0}}\xi_{1} t}e^{iu2^{-k}\xi_{2}\gamma_{l_{0}}(2^{l})\gamma_{l_{0}-l}(t)}\phi(t)\tilde{\varphi}(2^{l-l_{0}}t)\,\frac{dt}{t}\phi_{l_{0}}(\gamma_{l_{0}}(2^{l})^{-1}\xi_{1})\phi_{k}(\xi_{2})
		\end{align*}
		where $\gamma_{l_{0}}(2^{l})$ and $\gamma_{l_{0}-l}(t)$ are defined as in \eqref{c2}.
		By scaling, it suffices to show
		\begin{align}\label{d10}
			\bigg\|\sup_{ u\in[1,2]}\bigg(\sum_{k\in E_{v}}\big|m_{k}^{l}(u,D)f\big|^{2}\bigg)^{1/2}\bigg\|_{p}\lesssim (l+1)^{5} \|f\|_{p}.
		\end{align}
		Define $a_{k}^{l}(u,\xi)=\psi(u)m_{k}^{l}(u,2^{l_{0}}\gamma_{l_{0}}(2^{l})\xi_{1},2^{k}\xi_{2})$ for some $\psi(u)\in C_{c}^{\infty}(1/2,4)$ and $\psi=1$ on $[1,2]$. Then
		\begin{align}\label{b15}
			a_{k}^{l}(u,\xi)=\psi(u)\int_{\R} e^{i\gamma_{l_{0}}(2^{l})\xi_{1} t}e^{iu\xi_{2}\gamma_{l_{0}}(2^{l})\gamma_{l_{0}-l}(t)}\phi(t)\tilde{\varphi}(2^{l-l_{0}}t)\,\frac{dt}{t}\phi(\xi_{1})\phi(\xi_{2})
		\end{align}
		and
		$\supp a_{k}\subset I\times K$, where $I=\supp\psi$ and $K=\supp\phi\times \supp\phi\subset[1/2,2]\times[1/2,2]$.
		
		We have that, for all $u\in I,$ $\xi\in K$ and $k\in E_{v}$,
		\begin{align}\label{b18}
			|a_{k}^{l}(u,\xi)|\lesssim \gamma_{l_{0}}(2^{l})^{-1/2},
		\end{align}
		where we used van der Corput's lemma and Lemma \ref{b44}.
		
		Note $\frac{d}{du}a_{k}^{l}(u,\xi)$ can be written as the sum of the following two terms:
		\begin{align*}
			\psi'(u)\int_{\R} e^{i2^{j}\xi_{1} t}e^{iu\xi_{2}\gamma_{l_{0}}(2^{l})\gamma_{l_{0}-l}(t)}\phi(t)\tilde{\varphi}(2^{l-l_{0}}t)\,\frac{dt}{t}\phi(\xi_{1})\phi(\xi_{2}),
		\end{align*}
and
		\begin{align*}
		\gamma_{l_{0}}(2^{l})\psi(u)\int_{\R} e^{i2^{j}\xi_{1} t}e^{iu\xi_{2}\gamma_{l_{0}}(2^{l})\gamma_{l_{0}-l}(t)}\gamma_{-l_{0}+l}(t)\phi(t)\tilde{\varphi}(2^{l-l_{0}}t)\,\frac{dt}{t}\phi(\xi_{1})\xi_{2}\phi(\xi_{2}).
		\end{align*}
	By Lemma \ref{b44} for all $m\geq0$ and $t\in \supp\phi$, we have $|(\frac{d}{dt})^{m}(\gamma_{l_{0}-l}(t))|\lesssim 1$, where the implicit constant is independent of $l_{0}$ and $l$. Then similar as \eqref{b18}, we have for all $u\in I,$
		\begin{align}\label{b19}
			\bigg|\frac{d}{du}a_{k}^{l}(u,\xi)\bigg|\lesssim \gamma_{l_{0}}(2^{l})^{1/2}.
		\end{align}
		By \eqref{b18}, \eqref{b19} and Plancherel's theorem, we conclude that
		\begin{align*}
			\big\|a_{k}^{l}(u,D)f\big\|_{2}\lesssim\gamma_{l_{0}}(2^{l})^{-1/2}\|f\|_{2} \ \ \textrm{and} \ \ \bigg\|\frac{d}{du}a_{k}^{l}(u,D)f\bigg\|_{2}\lesssim\gamma_{l_{0}}(2^{l})^{1/2}\|f\|_{2}.
		\end{align*}
		By the following Sobolev embedding inequality,
		\begin{align*}
			\sup_{s\in I}|g(s)|^{2}\lesssim |g(1)|^{2}+\bigg(\int_{I}|g(s)|^{2}\,ds\bigg)\bigg(\int_{I}\big|\frac{d}{ds}g(s)\big|^{2}\,ds\bigg)
		\end{align*}
		where $g\in C^{\infty}(I)$, we get
		\begin{align}\label{b11}
			\bigg\|\sup_{u\in I}\big|a_{k}^{l}(u,D)f\big|\bigg\|_{2}\lesssim \|f\|_{2}.
		\end{align}
		By \eqref{b15}, it follows that
		\begin{align*}
			a_{k}^{l}(u,D)f=\psi(u)\int_{\R} P_{0}^{1}P_{0}^{2}f(x_{1}-2^{j}t,x_{2}-u\gamma_{l_{0}}(2^{l})\gamma_{l_{0}-l}(t))\phi(t)\tilde{\varphi}(2^{l-l_{0}}t)\,\frac{dt}{t}.
		\end{align*}
		Then we have $\sup_{x}\sup_{u\in I}|a_{k}^{l}(u,D)f(x)|\lesssim \|f\|_{\infty}.$
		We also have
		$|\partial_{\xi}^{\alpha}a_{k}^{l}(u,\xi)|\lesssim \gamma_{l_{0}}(2^{l})^{6l}\lesssim C_{0,U}^{6l}$
		 for all $u\in I,$ $\xi\in K$, $k\in E_{v}$ and $|\alpha|\leq 6$, where we used the doubling property of $\gamma$ in Lemma \ref{c1}.
		
		Therefore, by Lemma \ref{a13}, we conclude that
		\begin{align*}
			\bigg\|\sup_{u\in I}\bigg(\sum_{k\in E_{v}}\big|\psi(u)m_{k}^{l}(u,D)f\big|^{2}\bigg)^{1/2}\bigg\|_{p}\lesssim \big(1+\log_{2}(C_{0,L}^{7l})\big)^{\frac{1}{2}-\frac{1}{p}}\|f\|_{p}\lesssim (l+1)^{5} \|f\|_{p}.
		\end{align*}
		This along with the fact that $\psi=1$ on $[1,2]$ yields \eqref{d10}.
		
		Finally, we use \eqref{a40} to show \eqref{a41}.
		Define
		\begin{align*}
			\textrm{$J_{i}=\{3k+i:\ k\in \mathbb{Z}\}$,\ \ $i=1,2,3.$}
		\end{align*}
		Fix $i=1,2,3$ and define
		$g=\sum_{k\in J_{i}\bigcap E_{v}}\tilde{P}_{\kone}^{1}\tilde{P}_{k}^{2}f$ with $(\tilde{P}_{k}^{2}f)^\wedge(\xi)=\tilde{\phi}(2^{-k}\xi_{2})\widehat{f}(\xi)$,
	where $\tilde{\phi}\in C_{c}^{\infty}(1/4\leq|\xi_{2}|\leq4)$ and $\tilde{\phi}=1$ on $\{\xi_2\in \mathbb{R}  :\ 1/2\leq|\xi_2|\leq2\}$.
		By \eqref{a40}, we find that
		\begin{align*}
			\bigg\|\sup_{ u\in[1,2]}\bigg(\sum_{k\in J_{i}\bigcap E_{v}}\big|A_{2^{v}u,l_{0}-l}P_{\kone}^{1}P_{k}^{2}g\big|^{2}\bigg)^{1/2}\bigg\|_{p}\lesssim (l+1)^{5}\|g\|_{p}.
		\end{align*}
		Note that
		\begin{align*}
			P_{\kone}^{1}P_{k}^{2}g=P_{\kone}^{1}P_{k}^{2}f ~ \textrm{for~ all} ~k\in J_{i}, \ \ \textrm{and} \ \ \ \|g\|_{p}\lesssim \bigg\|\bigg(\sum_{k\in J_{i}\bigcap E_{v}}\big|\tilde{P}_{\kone}^{1}\tilde{P}_{k}^{2}f\big|^{2}\bigg)^{1/2}\bigg\|_{p}.
		\end{align*}
		For all $i=1,2,3,$ we have
		$$\bigg\|\sup_{ u\in[1,2]}\bigg(\sum_{k\in J_{i}\bigcap E_{v}}\big|A_{2^{v}u,l_{0}-l}P_{\kone}^{1}P_{k}^{2}f\big|^{2}\bigg)^{1/2}\bigg\|_{p}\lesssim (l+1)^{5}\bigg\|\bigg(\sum_{k\in J_{i}\bigcap E_{v}}\big|\tilde{P}_{\kone}^{1}\tilde{P}_{k}^{2}f\big|^{2}\bigg)^{1/2}\bigg\|_{p}.$$
		Then \eqref{a41} follows by the triangle inequality.
	
{\bf Acknowledgements.}	The authors are very grateful to Shaoming Guo for his constant support and many valuable discussions. The authors thank Lixin Yan and Liang Song for helpful suggestions. H.Y. is supported by the Guangdong Basic and Applied Basic Research Foundation (No. 2020A1515110241).

\bibliographystyle{plain}

\bibliography{HT}

\begin{thebibliography}{10}

\bibitem{MR3090145}
M.~Bateman.
\newblock Single annulus {$L^p$} estimates for {H}ilbert transforms along
  vector fields.
\newblock {\em Rev. Mat. Iberoam.}, 29(3):1021--1069, 2013.

\bibitem{MR3148061}
M.~Bateman and C.~Thiele.
\newblock {$L^p$} estimates for the {H}ilbert transforms along a one-variable
  vector field.
\newblock {\em Anal. PDE}, 6(7):1577--1600, 2013.

\bibitem{MR4078231}
D.~Beltran, J.~Hickman, and C.D. Sogge.
\newblock Variable coefficient {W}olff-type inequalities and sharp local
  smoothing estimates for wave equations on manifolds.
\newblock {\em Anal. PDE}, 13(2):403--433, 2020.

\bibitem{MR1926840}
J.M. Bennett.
\newblock Hilbert transforms and maximal functions along variable flat curves.
\newblock {\em Trans. Amer. Math. Soc.}, 354(12):481--4892, 2002.

\bibitem{MR1009171}
J.~Bourgain.
\newblock A remark on the maximal function associated to an analytic vector
  field.
\newblock In {\em Analysis at {U}rbana, {V}ol. {I} ({U}rbana, {IL},
  1986--1987)}, volume 137 of {\em London Math. Soc. Lecture Note Ser.}, pages
  111--132. Cambridge Univ. Press, Cambridge, 1989.

\bibitem{MR1046743}
A.~Carbery, M.~Christ, J.~Vance, S.~Wainger, and D.K. Watson.
\newblock Operators associated to flat plane curves: {$L^p$} estimates via
  dilation methods.
\newblock {\em Duke Math. J.}, 59(3):675--700, 1989.

\bibitem{MR1689214}
A.~Carbery and S.~P\'{e}rez.
\newblock Maximal functions and {H}ilbert transforms along variable flat
  curves.
\newblock {\em Math. Res. Lett.}, 6(2):237--249, 1999.

\bibitem{MR1296728}
A.~Carbery, J.~Vance, S.~Wainger, and D.~Watson.
\newblock The {H}ilbert transform and maximal function along flat curves,
  dilations, and differential equations.
\newblock {\em Amer. J. Math.}, 116(5):1203--1239, 1994.

\bibitem{MR1364881}
A.~Carbery, S.~Wainger, and J.~Wright.
\newblock Hilbert transforms and maximal functions along variable flat plane
  curves.
\newblock {\em J. Fourier Anal. Appl. Special Issue}, pages 119--19, 1995.

\bibitem{MR2948242}
J.~Chen and X.~Zhu.
\newblock {$L^2$}-boundedness of {H}ilbert transforms along variable curves.
\newblock {\em J. Math. Anal. Appl.}, 395(2):515--522, 2012.

\bibitem{MR1096400}
M.~Christ.
\newblock A {$T(b)$} theorem with remarks on analytic capacity and the {C}auchy
  integral.
\newblock {\em Colloq. Math.}, 60/61(2):601--628, 1990.

\bibitem{MR1726701}
M.~Christ, A.~Nagel, E.M. Stein, and S.~Wainger.
\newblock Singular and maximal {R}adon transforms: analysis and geometry.
\newblock {\em Ann. of Math. (2)}, 150(2):489--577, 1999.

\bibitem{MR0499948}
R.R. Coifman and G.~Weiss.
\newblock {\em Analyse harmonique non-commutative sur certains espaces
  homog\`enes}.
\newblock Lecture Notes in Mathematics, Vol. 242. Springer-Verlag, Berlin-New
  York, 1971.
\newblock \'{E}tude de certaines int\'{e}grales singuli\`eres.

\bibitem{MR813582}
A.~C\'{o}rdoba, A.~Nagel, J.~Vance, S.~Wainger, and D.~Weinberg.
\newblock {$L^p$} bounds for {H}ilbert transforms along convex curves.
\newblock {\em Invent. Math.}, 83(1):59--71, 1986.

\bibitem{MR3841536}
F.~Di~Plinio, S.~Guo, C.~Thiele, and P.~Zorin-Kranich.
\newblock Square functions for bi-{L}ipschitz maps and directional operators.
\newblock {\em J. Funct. Anal.}, 275(8):2015--2058, 2018.

\bibitem{MR209787}
E.B. Fabes and N.M. Rivi\`ere.
\newblock Singular integrals with mixed homogeneity.
\newblock {\em Studia Math.}, 27:19--38, 1966.

\bibitem{MR447953}
C.~Fefferman and E.M. Stein.
\newblock {$H^{p}$} spaces of several variables.
\newblock {\em Acta Math.}, 129(3-4):137--193, 1972.

\bibitem{MR3393679}
S.~Guo.
\newblock Hilbert transform along measurable vector fields constant on
  {L}ipschitz curves: {$L^2$} boundedness.
\newblock {\em Anal. PDE}, 8(5):1263--1288, 2015.

\bibitem{MR3592519}
S.~Guo.
\newblock Hilbert transform along measurable vector fields constant on
  {L}ipschitz curves: {$L^p$} boundedness.
\newblock {\em Trans. Amer. Math. Soc.}, 369(4):2493--2519, 2017.

\bibitem{MR3669936}
S.~Guo, J.~Hickman, V.~Lie, and J.~Roos.
\newblock Maximal operators and {H}ilbert transforms along variable non-flat
  homogeneous curves.
\newblock {\em Proc. Lond. Math. Soc. (3)}, 115(1):177--219, 2017.

\bibitem{MR161099}
B.F. Jones.
\newblock A class of singular integrals.
\newblock {\em Amer. J. Math.}, 86:441--462, 1964.

\bibitem{MR1013815}
P.W. Jones.
\newblock Square functions, {C}auchy integrals, analytic capacity, and harmonic
  measure.
\newblock In {\em Harmonic analysis and partial differential equations ({E}l
  {E}scorial, 1987)}, volume 1384 of {\em Lecture Notes in Math.}, pages
  24--68. Springer, Berlin, 1989.

\bibitem{MR2219012}
M.~Lacey and X.~Li.
\newblock Maximal theorems for the directional {H}ilbert transform on the
  plane.
\newblock {\em Trans. Amer. Math. Soc.}, 358(9):4099--4117, 2006.

\bibitem{MR2654385}
M.~Lacey and X.~Li.
\newblock On a conjecture of {E}. {M}. {S}tein on the {H}ilbert transform on
  vector fields.
\newblock {\em Mem. Amer. Math. Soc.}, 205(965):viii+72, 2010.

\bibitem{LiYu21}
J.~Li and H.~Yu.
\newblock {$L^2$} boundedness of {H}ilbert transforms along variable flat
  curves.
\newblock \emph{Math. Z.} (2021). https://doi.org/10.1007/s00209-020-02672-9.

\bibitem{NjSlHx20}
N.~Liu, L.~Song, and H.~Yu.
\newblock ${L}^{p}$ bounds of maximal operators along variable planar curves in
  the {L}ipschitz regularity.
\newblock {\em J. Funct. Anal.}, 280(5):108888, 2021.

\bibitem{MR1168960}
G.~Mockenhaupt, A.~Seeger, and C.D. Sogge.
\newblock Local smoothing of {F}ourier integral operators and
  {C}arleson-{S}j\"{o}lin estimates.
\newblock {\em J. Amer. Math. Soc.}, 6(1):65--130, 1993.

\bibitem{MR0466470}
A.~Nagel, E.M. Stein, and S.~Wainger.
\newblock Differentiation in lacunary directions.
\newblock {\em Proc. Nat. Acad. Sci. U.S.A.}, 75(3):1060--1062, 1978.

\bibitem{MR545242}
A.~Nagel, E.M. Stein, and S.~Wainger.
\newblock Hilbert transforms and maximal functions related to variable curves.
\newblock In {\em Harmonic analysis in {E}uclidean spaces ({P}roc. {S}ympos.
  {P}ure {M}ath., {W}illiams {C}oll., {W}illiamstown, {M}ass., 1978), {P}art
  1}, Proc. Sympos. Pure Math., XXXV, Part, pages 95--98. Amer. Math. Soc.,
  Providence, R.I., 1979.

\bibitem{MR714828}
A.~Nagel, J.~Vance, S.~Wainger, and D.~Weinberg.
\newblock Hilbert transforms for convex curves.
\newblock {\em Duke Math. J.}, 50(3):735--744, 1983.

\bibitem{MR955772}
A.~Seeger.
\newblock Some inequalities for singular convolution operators in
  {$L^p$}-spaces.
\newblock {\em Trans. Amer. Math. Soc.}, 308(1):259--272, 1988.

\bibitem{MR1258491}
A.~Seeger.
\newblock {$L^2$}-estimates for a class of singular oscillatory integrals.
\newblock {\em Math. Res. Lett.}, 1(1):65--73, 1994.

\bibitem{MR2053571}
A.~Seeger and S.~Wainger.
\newblock Singular {R}adon transforms and maximal functions under convexity
  assumptions.
\newblock {\em Rev. Mat. Iberoamericana}, 19(3):1019--1044, 2003.

\bibitem{MR934224}
E.M. Stein.
\newblock Problems in harmonic analysis related to curvature and oscillatory
  integrals.
\newblock In {\em Proceedings of the {I}nternational {C}ongress of
  {M}athematicians, {V}ol. 1, 2 ({B}erkeley, {C}alif., 1986)}, pages 196--221.
  Amer. Math. Soc., Providence, RI, 1987.

\bibitem{MR1232192}
E.M. Stein.
\newblock {\em Harmonic analysis: real-variable methods, orthogonality, and
  oscillatory integrals}, volume~43 of {\em Princeton Mathematical Series}.
\newblock Princeton University Press, Princeton, NJ, 1993.
\newblock With the assistance of Timothy S. Murphy, Monographs in Harmonic
  Analysis, III.

\bibitem{MR2880220}
E.M. Stein and B.~Street.
\newblock Multi-parameter singular {R}adon transforms {III}: {R}eal analytic
  surfaces.
\newblock {\em Adv. Math.}, 229(4):2210--2238, 2012.

\bibitem{MR508453}
E.M. Stein and S.~Wainger.
\newblock Problems in harmonic analysis related to curvature.
\newblock {\em Bull. Amer. Math. Soc.}, 84(6):1239--1295, 1978.

\end{thebibliography}

\end{document}